\documentclass{siamart251216}

\usepackage{bm}
\usepackage{tikz-cd}
\usepackage{amsfonts}
\usepackage{graphicx}
\usepackage{algorithmic}
\usepackage{amsmath,amssymb}
\usepackage{mathrsfs}
\usepackage{stmaryrd}
\usepackage{booktabs}
\usepackage{multirow}
\usepackage{placeins}

\newsiamremark{remark}{Remark}
\newtheorem{assumption}{Assumption}

\headers{Improving accuracy of FE solutions by smoothing}{Y. Li, H. Shui, and L. Zikatanov}

\title{Improving accuracy of finite element solutions by smoothing\thanks{The work of Li and Shui was supported by the National Key R\&D Program of China under grant 2024YFA1012600 and the National Natural Science Foundation of China under grant 12471346. The work of Zikatanov was supported in part by the Center for Advanced Studies at the Norwegian Academy of Science and Letters, Norway under the program Mathematical Challenges in Brain Mechanics (2025/2026) and by the Center for Computational Mathematics and its Applications at Penn State.}}

\author{
Yuwen Li\thanks{School of Mathematical Sciences, Zhejiang University, 866 Yuhangtang Road, Hangzhou, Zhejiang 310058, People's Republic of China 
(\email{liyuwen@zju.edu.cn}, \email{shuihan@zju.edu.cn}).}\and Han Shui\footnotemark[2]
\and
  Ludmil Zikatanov\thanks{Department of Mathematics, The {Pennsylvania} State University, University Park, PA 16802, USA (\email{ludmil@psu.edu}).}
}

\begin{document}
\maketitle

\begin{abstract}
This paper develops a smoothing-based postprocessing method for {improving the accuracy of finite element solutions}. The method applies a few smoothing iterations, such as damped Jacobi, symmetrized Gauss-Seidel, or preconditioned conjugate gradient, with {the} initial guess being the current finite element solution embedded in an augmented finite element space. The resulting procedure is algebraic, easy to implement, and applicable to high-order and three-dimensional discretizations. For symmetric and positive-definite problems, we {show that the accuracy of the smoothed solutions is improved by a smoothing factor} under additive and multiplicative smoothers. {This smoothing factor decays exponentially during the first few smoothing steps.} {The effectiveness} of the proposed method is numerically demonstrated using the Poisson, Maxwell, biharmonic and Helmholtz equations.
\end{abstract}



\section{Introduction}
Superconvergence in finite element (FE) methods has been under extensive investigation since the 1970s. Research results in this field can be divided into postprocessing-/recovery-type superconvergence and natural superconvergence. {The former uses} an operator $R$ to improve {the} accuracy of the FE solution $u_h$ approximating the exact solution $u$. A typical example is the gradient recovery superconvergence  $\|\nabla u-Ru_h\|\ll\|\nabla u-\nabla u_h\|$, see, e.g.,~\cite{ZienkiewiczZhu1992a,BankXu2003a,ZhangNaga2005}. The computational cost of $R$ is proportional to the number of degrees of freedom (dofs), e.g., $Ru_h$ is a local average or least-squares fitting of functional/derivative values  of $u_h$. On the other hand, natural superconvergence of $u-u_h$ happens at special points without employing any postprocessing in FE solutions. In either case, superconvergence analysis is highly sensitive to the structure of the underlying mesh, space dimensions and the type of Finite Element Methods (FEMs). Most postprocessing-type superconvergence results were proved on highly structured 2D grids, see, e.g., \cite{ZienkiewiczZhu1992a,BankXu2003a,ZhangNaga2005,Li2018SINUM,BankLi2019,DuWuZhang2020,HuMaMa2021}, while natural superconvergence is limited to second-order elliptic PDEs and Lagrange FEs on structured meshes (cf.~\cite{Wahlbin1995,LinZhang2008,HeLinZhang2016}). 

{To improve the accuracy of FE solutions}, we present a new postprocessing procedure based on convergent relaxation or Krylov space methods (smoothing) in a higher order FE space. The key module is a smoother $S$, which corresponds  to one step of Jacobi or Gauss-Seidel (GS) iteration for a higher order FE stiffness matrix. Our postprocessed FE solution $R_mu_h$ is obtained by $m$-step fixed-point or preconditioned conjugate gradient (PCG) iterations for the higher order FE linear system $A\tilde{u}_h=\tilde{f}$ with $u_h$ as the initial guess. {Here $m$ is chosen as a small, problem-dependent number.} We {derive estimates for quantifying the improved FE accuracy of} positive-definite problems such as the Poisson and Maxwell equations on quasi-uniform meshes. The implementation of the smoothing procedure is a two-grid method, which prolongates the current FE solution $u_h$ to an auxiliary higher order FE space and then performs simple smoothing passes. The cost of this process is linear in the number of dofs.

{Let $u_h$ and $\tilde{u}_h$ be the $k$-th and $(k+1)$-th order FE solutions under the energy norm $\|\bullet\|_a$, respectively. We establish an error estimate of the form
\begin{equation}\label{eq:prelim}
    \|u-R_mu_h\|_a=\mathcal{O}(h^{k+1})+\varepsilon_m\mathcal{O}(h^k).
\end{equation}
Here $\varepsilon_m$ is an $m$-step smoothing factor, which decays exponentially during the first few smoothing passes. If $\varepsilon_m=\mathcal{O}(h^\alpha)$ for some $\alpha>0$, then \eqref{eq:prelim} implies $\|u-R_mu_h\|_a=\mathcal{O}(h^{k+\min(\alpha,1)})$. Within the initial exponential-decay regime, this condition can be achieved with a number of smoothing steps proportional to $|\log h|$.  
However, the available bound for $\varepsilon_m$ does not guarantee exponential decay uniformly in $h$ for all $m$, so achieving $\varepsilon_m=\mathcal{O}(h^\alpha)$ on very fine meshes may require many smoothing steps.  Nevertheless, extensive numerical experiments in Section \ref{sect:numerical} show that the smoothing-based method produces apparent superconvergence  for a practical mesh-size $h$.} 

Improving FE accuracy by smoothing was initiated in the seminal work \cite{BankXu2003b} and generalized to high-order and $h$-$p$ FEs in \cite{BankXuZheng2007,BankNguyen2011}. The smoothing approach developed in \cite{BankXu2003b,BankXuZheng2007,BankNguyen2011} is dependent on hierarchical bubble decomposition of FE spaces and is devoted to recovery of derivatives of Lagrange-type FE solutions in 2D, e.g., gradient and Hessian recovery. It is not clear how to extend the theoretical results or even the numerical algorithms in \cite{BankXu2003b,BankXuZheng2007,BankNguyen2011} to higher space dimensions and other popular FEs. {In contrast, our method postprocesses the solution itself in an augmented FE space, without requiring an explicit hierarchical structure. It applies to arbitrary order FEs in 3D, including Nédélec's edge FEs. }

Improving FE accuracy by postprocessing is useful in adaptive methods because $\|u_h-Ru_h\|_a$ can serve as an accurate a posteriori error estimate for guiding local mesh refinement. {Assuming the lower bound $\|u-u_h\|_a\gtrsim h^k$, we have $$\left|\frac{\|u_h-R_mu_h\|_a}{\|u-u_h\|_a}-1\right|\lesssim h+\varepsilon_m$$ and numerically observe a sharp effectivity ratio close to 1 on locally refined meshes in Section~\ref{subsect:adapt}.}  Our a posteriori error estimate based on smoothing $R_m$ is a $p$-variant of the $h$-$h/2$ smoother-type error estimator in \cite{LiShui2026}, while the ones in \cite{LiShui2026} are not asymptotically exact in {practice}. Readers are referred to \cite{BankSmith1993,Bank1996,MulitaGianiHeltai2021,LiZikatanov2021CAMWA,LiZikatanov2027mcom} for other interesting a posteriori error estimates and adaptive algorithms motivated by linear iterative solvers.

The rest of the paper is organized as follows. In
Section~\ref{sect:framework}, we set up the abstract framework and derive main theoretical results. In Section~\ref{sect:examples}, we present applications of {smoothing-based improved accuracy} for continuous and discontinuous FE methods for the Poisson, Maxwell, and biharmonic equations. Section~\ref{sect:numerical} illustrates the numerical effectiveness of the proposed methods. Concluding remarks are given in Section~\ref{sect:conclusion}.

\section{Framework of Smoothing Superconvergence}\label{sect:framework}
In this section, we present a framework for {improved FE accuracy} by smoothing passes. 

\subsection{Abstract Setting}
Let $\mathcal{T}_h$ be a quasi-uniform triangulation of a domain $\Omega\subset\mathbb{R}^d$ with mesh-size $h$ and $V$ be an FE space based on $\mathcal{T}_h$. Let $a: V\times V\rightarrow \mathbb{R}$ be a symmetric and coercive bilinear form. Let $\|v\|_a=\sqrt{a(v,v)}$ denote the energy norm of $v\in V$. Let $V^\prime$ denote the dual space consisting of continuous linear functionals on $V$. By $T^\prime: W^\prime\rightarrow V^\prime$ we mean the Banach adjoint of a continuous linear operator $T: V\rightarrow W$ between Banach spaces.

Given $f\in V^\prime$, an FE discretization seeks $u_h\in V$ such that
\begin{equation}\label{eq:FEM}
    a(u_h,v_h)=f(v_h),\quad \forall v_h\in V.
\end{equation}
Here $u_h$ approximates the true solution $u$ of some PDE model. For obtaining superconvergence, we shall make use of an augmented FE space $\widetilde{V}\supset V$ with $\dim\widetilde{V}=\widetilde{N}$. For theoretical analysis, we consider the augmented FE solution $\tilde{u}_h\in \widetilde{V}$ satisfying
\begin{equation}\label{eq:FEM_tilde}
a(\tilde{u}_h,v_h)=\tilde{f}(v_h),\quad \forall v_h\in\widetilde{V},  
\end{equation}
where $\tilde{f}\in\widetilde{V}^\prime$ with $\tilde{f}|_{V}=f$ and $a$ is extended to $\widetilde{V}$. {In the PDE applications below, $f$ and $\tilde{f}$ are obtained by restricting the same load functional to $V$ and $\widetilde {V}$, respectively. Together with the common bilinear form, this gives $a(\tilde{u}_h-u_h,v_h)=0$ for all $v_h\in V$, so $u_h$ is the $a$-orthogonal projection of $\tilde{u}_h$ onto $V$.}

Let $\langle\bullet,\bullet\rangle$ denote the duality pairing between $\widetilde{V}^\prime$ and $\widetilde{V}$. We say $B: \widetilde{V}\rightarrow \widetilde{V}^\prime$ is symmetric and positive-definite (SPD) provided $\langle Bv_1,v_2\rangle=\langle Bv_2,v_1\rangle$ and $\langle Bv_3,v_3\rangle>0$ for any $v_1, v_2\in\widetilde{V}$, $0\neq v_3\in\widetilde{V}$. The $B$-inner product of $\widetilde{V}$ is $(\bullet,\bullet)_B=\langle B\bullet,\bullet\rangle$. The symmetry and positive-definiteness of 
$B: \widetilde{V}^\prime\rightarrow \widetilde{V}$ are defined in a similar way.

We define the operator $A: \widetilde{V}\rightarrow \widetilde{V}^\prime$ in a standard fashion using the symmetric and positive definite bilinear form: $a: \widetilde{V}\times \widetilde{V}\rightarrow\mathbb{R}$, namely, 
\begin{equation*}   \langle Av_h,w_h\rangle=a(v_h,w_h),\quad v_h, w_h\in \widetilde{V}.
\end{equation*}
We introduce a smoother $S: \widetilde{V}^\prime\rightarrow \widetilde{V}$, which, in the simplest case, corresponds to the Jacobi or GS iteration for the augmented problem \eqref{eq:FEM_tilde}. Our goal is to postprocess $u_h$ using several steps of smoothing as outlined in Algorithm \ref{alg:smoothing}.

\begin{algorithm}[thbp]
\caption{Postprocessing $R_m$ by simple smoothing}\label{alg:smoothing}
\begin{algorithmic}
\STATE \textbf{Input}: $u_h\in V$, {$A: \widetilde{V}\rightarrow\widetilde{V}^\prime$, $\tilde{f}\in\widetilde{V}^\prime$}, a smoother $S: \widetilde{V}^\prime\rightarrow\widetilde{V}$;
\STATE set $u_0=u_h$;
\FOR{$k=1:m$}
    \STATE $u_k=u_{k-1}+S(\tilde{f}-Au_{k-1})$;
\ENDFOR

\STATE \textbf{Output}: $R_mu_h=u_m$.

\end{algorithmic}
\end{algorithm}

The cost of one action of $S$ is {$\mathcal{O}(\widetilde{N})$ with $\widetilde{N}=\dim\widetilde{V}$.} The total cost of $R_mu_h$ is {$\mathcal{O}(m\widetilde{N})$}, so the postprocessing cost is linear in the number of dofs for fixed \(m\). {Here the polynomial degree of $\widetilde{V}$ is fixed, and the cost of local patch solves is included in the hidden constant.} In the matrix-vector notation, Algorithm \ref{alg:smoothing} translates into 
\begin{equation*}
  \mathbf{u}_k=\mathbf{u}_{k-1}+\mathbf{S}(\tilde{\mathbf{f}}-\mathbf{A}\mathbf{u}_{k-1}),  
\end{equation*}
where $\mathbf{u}_k$ is the coordinate vector of $u_k$, and $\mathbf{A}$ denotes the stiffness matrix for $A$. Let $\mathbf{A}=\mathbf{D}-\mathbf{L}-\mathbf{L}^\top$ be the usual diagonal-triangular splitting of $\mathbf{A}$. The Jacobi or GS smoother $S$ is represented by the matrix $\mathbf{S}=\mathbf{D}^{-1}$ or $\mathbf{S}=(\mathbf{D}-\mathbf{L})^{-1}$, respectively, see Section~\ref{sect:examples} for details.

\begin{remark}
{Algorithm~\ref{alg:smoothing} is written in a residual-update form, which covers stationary linear solvers,  such as defect-correction, two-level, hierarchical basis and even multigrid methods, for $A\tilde{u}_h=\tilde{f}$. However, our methodological viewpoint is different from that of linear solvers, which are intended to approximate the augmented solution $\tilde{u}_h$ to a prescribed algebraic accuracy. In our setting, the lower-order FE solution $u_h$ is already available, and we deliberately apply only a fixed, small number of simple smoothing steps, such as Jacobi or GS iteration, as a postprocessing procedure. Unlike two-level or multi-level methods, we never require expensive or complicated coarse-level corrections. The resulting $R_mu_h$ need not be an accurate approximation of $\tilde u_h$. In fact, Figure~\ref{fig:P1P2_GS_CG} shows that a few smoothing steps remain far from solving the augmented system accurately.

From the viewpoint of superconvergence, each smoothing pass acts as an algebraic local-recovery process: patchwise corrections combine neighboring information and damp the high-frequency component of $\tilde u_h-u_h$. Our contribution is to relate this smoothing error to the accuracy of the postprocessed FE solution $\|u-R_mu_h\|_a$.}
\end{remark}

An alternative strategy is applying $m$ steps of preconditioned conjugate gradient iterations with initial guess $u_0=u_h$ and a preconditioner $S: \widetilde{V}^\prime\rightarrow\widetilde{V}$ to construct $R_mu_h$, see Algorithm~\ref{alg:PCG}. On the matrix level, the coordinate vector of $R_mu_h$ is the output of $m$-step PCG iterations for $\mathbf{A}\tilde{\mathbf{u}}=\tilde{\mathbf{f}}$ with preconditioner $\mathbf{S}$ and initial guess being the prolongated coordinate vector of $u_h$.

\begin{algorithm}[thbp]
\caption{Postprocessing $R_m$ by PCG smoothing}\label{alg:PCG}
\begin{algorithmic}
\STATE \textbf{Input}: $u_h\in V$, {$\tilde{f}\in\widetilde{V}^\prime$}, SPD operators {$A: \widetilde{V}\rightarrow\widetilde{V}^\prime$},  $S: \widetilde{V}^\prime\rightarrow\widetilde{V}$;
\STATE set $u_0=u_h$, $r_0=\tilde{f}-Au_h$, $p_0=z_0=Sr_0$;
\FOR{$k=1:m$}
 \STATE $\alpha_k=\langle r_{k-1},z_{k-1}\rangle/\langle Ap_{k-1},p_{k-1}\rangle$;
     \STATE $u_k=u_{k-1}+\alpha_kp_{k-1}$;
     \STATE $r_k=r_{k-1}-\alpha_kAp_{k-1}$;
     \STATE $z_k=Sr_k$;
     \STATE $\beta_k=\langle r_k,z_k\rangle/\langle r_{k-1},z_{k-1}\rangle$;
     \STATE $p_k=z_k+\beta_kp_{k-1}$;
 \ENDFOR

\STATE \textbf{Output}: $R_mu_h=u_m$.

\end{algorithmic}
\end{algorithm}

\subsection{Main Theoretical Results}
Note that $SA: \widetilde{V}\rightarrow\widetilde{V}$ is SPD with respect to the inner product $(\bullet,\bullet)_{S^{-1}}$. Therefore, $SA$ admits $\widetilde{N}=\dim\widetilde{V}$ positive eigenvalues $0<\lambda_1\leq\lambda_2\leq\cdots\leq\lambda_{\widetilde{N}}$. Let $\psi_1, \psi_2, \ldots, \psi_{\widetilde{N}}\in\widetilde{V}$ be the corresponding orthonormal eigenfunctions of $SA$, i.e., $A\psi_i=\lambda_iS^{-1}\psi_i$ and $(\psi_i,\psi_j)_{S^{-1}}=\delta_{ij}$.
{Let $\sigma(SA)$ be the spectrum of $SA$. For a scalar function $p:\sigma(SA)\to\mathbb R$, define }
\[{p(SA)v=\sum_{i=1}^{\widetilde N}p(\lambda_i)v_i\psi_i,\qquad v=\sum_{i=1}^{\widetilde N}v_i\psi_i.}\] 
{This spectral expansion gives
\[
\|p(SA)v\|_{S^{-1}}\leq\Big(\max_{t\in\sigma(SA)}|p(t)|\Big)\|v\|_{S^{-1}}.
\]}

Now we are in a position to present our main theoretical result. 
\begin{theorem}\label{thm:main}
Let $S: \widetilde{V}^\prime\rightarrow\widetilde{V}$ be SPD and the following assumptions be true: 
\begin{subequations}
\begin{align}
\lambda_{\max}(SA)&\leq 1,\label{eq:SA_spectrum}\\
 \|\tilde{u}_h-u_h\|_{S^{-1}}&\leq \sqrt{\delta}\|\tilde{u}_h-u_h\|_a, \label{eq:high_frequency}
 \end{align}
\end{subequations}
where $\delta>1$ is a constant.
Then for Algorithm \ref{alg:smoothing}, it holds that 
\begin{equation*}
\|u-R_mu_h\|_a\leq2\|u-\tilde{u}_h\|_a+\varepsilon_m(\delta)\|u-u_h\|_a,
\end{equation*}  
where $\varepsilon_m(\delta)$ is defined as 
\begin{equation*}
\varepsilon_m(\delta)=\left\{\begin{aligned}
        &\Big(\frac{\delta-1}{\delta}\Big)^m,\quad {m\leq\delta/2},\\
        &{\sqrt{\frac{\delta-1}{2m}}\Big(\frac{2m-1}{2m}\Big)^{m-\frac12}},\quad {m>\delta/2}.
    \end{aligned}\right.
\end{equation*}
\end{theorem}
\begin{proof}
Recall that $\|\tilde{u}_h-R_mu_h\|_a$ is the error of the iterative method in Algorithm \ref{alg:smoothing} for $A\tilde{u}_h=\tilde{f}$ based on the smoother $S$ and the initial guess $u_h$. The smoothing error is related to the initial error $v:=\tilde{u}_h-u_h$ via \begin{equation*}
    \tilde{u}_h-R_mu_h=(I-SA)^mv.
\end{equation*}
Using a triangle inequality, we have
\begin{equation}\label{eq:triangle}
\begin{aligned}
\|u-R_mu_h\|_a&\leq\|u-\tilde{u}_h\|_a+\|\tilde{u}_h-R_mu_h\|_a\\
&\leq\|u-\tilde{u}_h\|_a+\|(I-SA)^mv\|_a.
\end{aligned}    
\end{equation}

{We write $v=\sum_{i=1}^{\widetilde N}v_i\psi_i$ and use orthonormality and $A\psi_i=\lambda_iS^{-1}\psi_i$ to obtain 
\[
\|v\|_{S^{-1}}^2=\sum_i v_i^2,\qquad \|v\|_a^2=\sum_i\lambda_i v_i^2.
\]
By \eqref{eq:SA_spectrum}, $\sigma(SA)\subset(0,1]$. For $m>\delta/2$, the identity $\max_{t\in[0,1]}t(1-t)^{2m-1}=(2m-1)^{2m-1}/(2m)^{2m}$ and \eqref{eq:high_frequency} give
\begin{align*}
\|(I-SA)^mv\|_a^2
&=\sum_i\lambda_i(1-\lambda_i)^{2m}v_i^2\\
&\leq\frac{(2m-1)^{2m-1}}{(2m)^{2m}}\sum_i(1-\lambda_i)v_i^2\\
&=\frac{(2m-1)^{2m-1}}{(2m)^{2m}}
\bigl(\|v\|_{S^{-1}}^2-\|v\|_a^2\bigr)\\
&\leq\varepsilon_m(\delta)^2\|v\|_a^2.
\end{align*}
For $1\leq m\leq\delta/2$, let $g_m(t)=(1-t^{-1})^{2m}$ and $L_\delta(t)=g_m(\delta)+g_m'(\delta)(t-\delta)$. Since
\[
g_m''(t)=\frac{2m(t-1)^{2m-2}}{t^{2m+2}}(2m+1-2t),
\]
$g_m$ is convex on $[1,m+\tfrac12]$ and concave on $[m+\tfrac12,\infty)$. Since $\delta\geq2m$, concavity gives $L_\delta(t)\geq g_m(t)$ for $t\geq m+\tfrac12$. Moreover, $L_\delta(1)=g_m(\delta)\Big(1-\frac{2m}{\delta}\Big)\geq0=g_m(1)$.
On $[1,m+\tfrac12]$, $L_\delta-g_m$ is concave and nonnegative at both endpoints, therefore $L_\delta(t)-g_m(t)\geq0$ for $t\in[1,m+\tfrac12]$. In what follows, 
\[
g_m(t)\leq g_m(\delta)+g_m'(\delta)(t-\delta),\qquad t\geq1,
\]
when $\delta\geq2m$. Since $g_m'(\delta)\geq0$, the spectral expansion, $\lambda_i^{-1}\geq1$ and \eqref{eq:high_frequency} yield
\begin{align*}
\|(I-SA)^mv\|_a^2
&=\sum_i\lambda_i v_i^2g_m(\lambda_i^{-1})\\
&\leq g_m(\delta)\|v\|_a^2+g_m'(\delta)
\bigl(\|v\|_{S^{-1}}^2-\delta\|v\|_a^2\bigr)\\
&\leq\Big(\frac{\delta-1}{\delta}\Big)^{2m}\|v\|_a^2.
\end{align*}
Thus $\|(I-SA)^mv\|_a\leq\varepsilon_m(\delta)\|v\|_a$ for all integers $m\geq0$, with $m=0$ being immediate. Combining this bound with \eqref{eq:triangle} yields}
\begin{align*}
\|u-R_mu_h\|_a&\leq\|u-\tilde{u}_h\|_a+\varepsilon_m(\delta)\|v\|_a\\
&\leq\|u-\tilde{u}_h\|_a+\varepsilon_m(\delta)\|u-\tilde{u}_h\|_a+\varepsilon_m(\delta)\|u-u_h\|_a.
\end{align*}
Noting that $\varepsilon_m(\delta)\leq1$ completes the proof. 
\end{proof}
{There is no loss of generality in requiring $\delta>1$, since if \eqref{eq:high_frequency} holds with $\delta=1$, then it also holds with any larger constant $\delta>1$. Following the convention in the multigrid literature \cite{BankDouglas1985,BankXu2003b}, we call $\varepsilon_m(\delta)$ the smoothing factor, which is a non-increasing function of $m$. In fact, our $\varepsilon_m(\delta)$ is slightly sharper than the one used in \cite{BankDouglas1985,BankXu2003b}.}  Due to Theorem \ref{thm:main}, we need to set $\widetilde{V}$ as a higher order FE space to obtain {an} improved order of convergence for $\|u-R_mu_h\|_a$. The factor $\varepsilon_m(\delta)$ decays exponentially during the first several smoothing steps, due to the efficiency of $S$ for damping the high frequency function $\tilde{u}_h-u_h$.

\begin{remark}
A natural question is whether the following error bound holds true for {an arbitrary number of smoothing steps} $m\in\mathbb{N}$:
\begin{equation*}
\|u-R_mu_h\|_a\lesssim\|u-\tilde{u}_h\|_a+\alpha^m\|u-u_h\|_a,
\end{equation*}
where $0<\alpha<1$ is a uniform contraction factor. The answer is negative because {the smoothing factor \(\varepsilon_m\) decays exponentially only over a problem-dependent initial range, which in Theorem~\ref{thm:main} is $m\leq\delta/2$}. In fact, it may take thousands of smoothing passes for $R_mu_h$ to be very close to the higher order FE solution $\tilde{u}_h$, see Figure \ref{fig:P1P2_GS_CG}.
\end{remark}

Let $\widetilde{V}=\sum_{j=1}^J\widetilde{V}_j$ be a subspace decomposition. Here we consider {smoothers given either by} the parallel subspace correction method (Algorithm \ref{alg:Sa}) or by the successive subspace correction method (Algorithm \ref{alg:Sm}), see \cite{Xu1992,XuZikatanov2002}.  

\begin{algorithm}[thbp]
\caption{Additive smoother $S_{\rm add}$}\label{alg:Sa}
\begin{algorithmic}
\STATE \textbf{Input}:  $r\in \widetilde{V}^\prime$;

\FOR{$j=1:J$}
    \STATE find $e_j\in \widetilde{V}_j$ such that $a(e_j,v_j)=r(v_j),\quad\forall v_j\in \widetilde{V}_j$;
\ENDFOR

\STATE \textbf{Output}:  $S_{\rm add}r=\sum_{j=1}^Je_j$.

\end{algorithmic}
\end{algorithm}

\begin{algorithm}[thbp]
\caption{Multiplicative smoother $S_{\rm mult}$}\label{alg:Sm}
\begin{algorithmic}
\STATE \textbf{Input}:  $r\in \widetilde{V}^\prime$;

\STATE set $e_0=0$;

\FOR{$j=1:J$}
    \STATE find $\eta_j\in \widetilde{V}_j$ such that $a(\eta_j,v_j)=r(v_j)-a(e_{j-1},v_j),\quad\forall v_j\in \widetilde{V}_j$;
\STATE set $e_j=e_{j-1}+\eta_j$;
\ENDFOR

\STATE \textbf{Output}: $S_{\rm mult}r=e_J$.

\end{algorithmic}
\end{algorithm}

By reversing the order of the for-loop in Algorithm \ref{alg:Sm}, we obtain a backward multiplicative smoother (denoted by $S_{\rm mult}^t$). A combination of $S_{\rm mult}$ and $S_{\rm mult}^t$ yields a symmetrized multiplicative smoother $\bar{S}_{\rm mult}: \widetilde{V}^\prime\rightarrow\widetilde{V}$ as follows:
\begin{equation}\label{eq:Sbar_mult}
    \bar{S}_{\rm mult}=S_{\rm mult}+S_{\rm mult}^t-S_{\rm mult}^tAS_{\rm mult}.
\end{equation}

For each $1\leq j\leq J$, let $I_j: \widetilde{V}_j\rightarrow \widetilde{V}$ be the inclusion. Each $\widetilde{V}_j$ corresponds to an SPD operator $A_j: \widetilde{V}_j\rightarrow \widetilde{V}_j^\prime$ with  $\langle A_jv_j,w_j\rangle=a(v_j,w_j),~\forall v_j, w_j\in \widetilde{V}_j$.
The additive smoother in Algorithm \ref{alg:Sa} is written as $S_{\rm add}=\sum_{j=1}^JI_jA_j^{-1}I_j^\prime$,
which satisfies the well-known additive preconditioning formula (cf.~\cite{BrennerScott2008})
\begin{equation}\label{eq:Sadd_inverse}
    \langle S_{\rm add}^{-1}v,v\rangle= \inf_{\sum_{j=1}^Jv_j=v}\sum_{j=1}^J\langle A_jv_j,v_j\rangle=\inf_{\sum_{j=1}^Jv_j=v}\sum_{j=1}^J\|v_j\|^2_a.
\end{equation}
The infimum is taken over all possible decompositions of the form  $\sum_{j}v_j$ with each $v_j\in \widetilde{V}_j$. Let $P_j$ be the $a(\bullet,\bullet)$-projection onto $\widetilde{V}_j$. It is shown in \cite{Zikatanov2008} that 
\begin{equation}\label{eq:Sm_operator}
    \langle \bar{S}_{\rm mult}^{-1}v,v\rangle=\inf_{\substack{\sum_{j=1}^Jv_j=v\\v_j\in \widetilde{V}_j}}\sum_{i=1}^J\Big\|P_i\sum_{k\geq i}v_k\Big\|^2_a.
\end{equation}
The analysis for $\bar{S}_{\rm mult}$ relies on the quantity
\[
M:=\max_{1\leq j\leq J}\#\big\{1\leq i\leq J: P_i\widetilde{V}_j\neq\{0\}\big\},
\]
which will be determined by the mesh regularity when solving concrete PDEs. 

We formulate the next corollary based on two assumptions about subspace decomposition, which {are} commonly used in the literature and can be easily verified case by case. 
\begin{assumption}\label{assump:stable}
Let $\widetilde{V}=V\oplus V^\perp$ where $V^\perp$ is the orthogonal complement under $a(\bullet,\bullet)$. There exists a constant $C_{\rm st}>0$ such that
\begin{equation*}
        \forall v\in \widetilde{V}\cap V^\perp,~\exists v_j\in \widetilde{V}_j\text{ s.t. }     v=\sum_{j=1}^Jv_j~\&~\sum_{j=1}^J\|v_j\|_a^2\leq C_{\rm st}\|v\|_a^2.
\end{equation*}
\end{assumption}

\begin{assumption}\label{assump:regular}
There exists a constant $C_{\rm reg}>0$ such that
\begin{equation*}
\forall v\in \widetilde{V},~\forall v_j\in \widetilde{V}_j\text{ with }v=\sum_{j=1}^Jv_j\Longrightarrow\|v\|_a^2\leq C_{\rm reg}\sum_{j=1}^J\|v_j\|^2_a.
\end{equation*}
\end{assumption}

\begin{corollary}\label{cor:main}
Let Assumptions \ref{assump:stable} and \ref{assump:regular} be true. 
Let  $\omega\in(0,C_{\rm reg}^{-1}]$ be a damping factor.
Then for Algorithm \ref{alg:smoothing} with $S=\omega S_{\rm add}$ or $S=\bar{S}_{\rm mult}$, it holds that
\begin{equation*}
\|u-R_mu_h\|_a\leq2\|u-\tilde{u}_h\|_a+\varepsilon_m\|u-u_h\|_a,
\end{equation*}
where $\varepsilon_m=\varepsilon_m(\delta)$ is defined in Theorem \ref{thm:main} with $\delta=C_{\rm st}/\omega$ for $S=\omega S_{\rm add}$ and $\delta=C_{\rm st}M^2$ for $S=\bar{S}_{\rm mult}$. 
\end{corollary}
\begin{proof}
First we consider the case $S=\omega S_{\rm add}$. By definition, $u_h$ is the $a$-orthogonal projection of $\tilde{u}_h$ onto $V$ and thus $v:=\tilde{u}_h-u_h\in \widetilde{V}\cap V^\perp$. Using Assumption \ref{assump:stable} and the formula \eqref{eq:Sadd_inverse},
we have  
\begin{equation*}        \|v\|_{S^{-1}}^2=\omega^{-1}\langle S_{\rm add}^{-1}v,v\rangle =\omega^{-1} \inf_{\sum_jv_j=v}\sum_{j=1}^J\|v_j\|_a^2 \leq\omega^{-1}C_{\rm st}\|v\|_a^2,
\end{equation*}
which verifies the assumption \eqref{eq:high_frequency} in Theorem \ref{thm:main} with $\delta=\omega^{-1}C_{\rm st}$. On the other hand, for any $w\in\widetilde{V}$, combining Assumption \ref{assump:regular} with \eqref{eq:Sadd_inverse} yields 
\begin{equation*}
  \langle Aw,w\rangle=\|w\|_a^2
  \leq C_{\rm reg}\inf_{\sum_jw_j=w}\sum_{j=1}^J\|w_j\|^2_a
  =C_{\rm reg}\langle S_{\rm add}^{-1}w,w\rangle.
\end{equation*}
{It follows that} $\lambda_{\max}(S_{\rm add}A)\leq C_{\rm reg}$ and the assumption \eqref{eq:SA_spectrum} is true.

Second, we consider the case $S=\bar{S}_{\rm mult}$. {Using Assumption}~\ref{assump:stable} and the multiplicative preconditioning formula \eqref{eq:Sm_operator}, we have 
\begin{equation*}
    \begin{aligned}
        \|v\|_{S^{-1}}^2&=\inf_{\sum_jv_j=v}\sum_{i=1}^J\Big\|P_i\sum_{k\geq i}v_k\Big\|^2_a\\
        &\leq M^2\inf_{\sum_jv_j=v}\sum_{j=1}^J\|v_j\|_a^2\leq M^2C_{\rm st}\|v\|_a^2
    \end{aligned}
\end{equation*}
Therefore, the assumption \eqref{eq:high_frequency} with $S=\bar{S}_{\rm mult}$ still holds. For the multiplicative smoother, the eigenvalue estimate $\lambda_{\max}(\bar{S}_{\rm mult}A)<1$ is always true due to the contraction $\|I-\bar{S}_{\rm mult}A\|_A<1$. 

Therefore, we have verified the two assumptions in Theorem \ref{thm:main} for both additive and multiplicative smoothers. The proof then follows from Theorem \ref{thm:main}.
\end{proof}

\begin{remark}
{
The forward Gauss--Seidel smoothing process $S_{\rm mult}$ is nonsymmetric and
therefore is not covered by Theorem~\ref{thm:main}. A separate analysis can
be developed at the cost of substantially lengthening and interrupting the present exposition. We
therefore restrict both the theoretical analysis and numerical experiments to the symmetrized smoother $\bar S_{\rm mult}$.
}
\end{remark}

A drawback of $\bar{S}_{\rm mult}$ is its successive nature that prevents parallel implementation. On the other hand, the cheaper and parallel operator $S_{\rm add}$  requires a sufficiently small damping factor $\omega$, while the optimal value of $\omega$ is generally not clear. The next theorem shows that $S_{\rm add}$ combined with PCG leads to smoothing-based high accuracy without using a damping factor.  
\begin{theorem}\label{thm:PCG}
Let $S: \widetilde{V}^\prime\rightarrow\widetilde{V}$ be SPD with $\lambda_{\max}(SA)\leq\lambda$ and assume that \begin{equation*}
 \|\tilde{u}_h-u_h\|_{S^{-1}}\leq \sqrt{\delta}\|\tilde{u}_h-u_h\|_a, 
\end{equation*}
where $\delta>0$ is a constant. Then for Algorithm \ref{alg:PCG} and $\lambda\delta>1$, it holds that 
\begin{equation*}
\|u-R_mu_h\|_a\leq{2}\|u-\tilde{u}_h\|_a+{\varepsilon_m^{\rm PCG}(\lambda\delta)}\|u-u_h\|_a,
\end{equation*} 
where {the PCG smoothing factor is defined by} 
\begin{equation*}
{\varepsilon_m^{\rm PCG}(\delta)
:=\min\left\{\varepsilon_m(\delta),\frac{\sqrt{\delta}}{2m+1}\right\}\leq1.}
\end{equation*}
\end{theorem}
\begin{proof} 
In view of the proof of Theorem \ref{thm:main}, it suffices to estimate $\tilde{u}_h-R_mu_h$ with $R_mu_h$ being the output of the $m$-step PCG iterations for $A\tilde{u}_h=\tilde{f}$ with initial guess $u_h$. Let $r_0=\tilde{f}-Au_h\in \widetilde{V}^\prime$ and $v=\tilde{u}_h-u_h=A^{-1}r_0$. The PCG iterate $R_mu_h$ solves the minimization problem (cf.~\cite{LiZikatanovZuo2024SISC,LiZikatanovZuo2025})
\begin{equation*}
    \|\tilde{u}_h-R_mu_h\|_a
    =\min_{w\in u_h+\mathcal{K}_m}\|\tilde{u}_h-w\|_a
\end{equation*}
over the Krylov subspace  \begin{align*}
\mathcal{K}_m&={\rm span}\{Sr_0,(SA)Sr_0,\ldots,(SA)^{m-1}Sr_0\}\\
&={\rm span}\{SAv,(SA)^2v,\ldots,(SA)^mv\}.
\end{align*}
{It follows that}
\begin{equation}\label{eq:CGerror}
\begin{aligned}
\|\tilde{u}_h-R_mu_h\|_a
    &=\min_{p_m(0)=1}\|p_m(SA)v\|_A\\
    &=\min_{p_m(0)=1}\|(SA)^{1/2}p_m(SA)v\|_{S^{-1}}\\&\leq\min_{p_m(0)=1}\max_{t\in(0,\lambda]}|t^{1/2}p_m(t)|\,\|v\|_{S^{-1}}\\
    &\leq(\lambda\delta)^{1/2}\min_{p_m(0)=1} \max_{t\in[0,1]}|t^{1/2}p_m(t)|\,\|v\|_a,
\end{aligned}
\end{equation}
where $p_m$ is taken over all polynomials of degree $\leq m$. The optimizer of the mini-max problem in \eqref{eq:CGerror} is $\hat{p}_m(t)=\frac{T_{2m+1}(\sqrt{t})}{(-1)^m(2m+1)\sqrt{t}}$ with $T_k$ being the $k$-th degree Chebyshev polynomial of the first kind. The optimum is $\min_{p_m(0)=1} \max_{t\in[0,1]}|t^{1/2}p_m(t)|=1/(2m+1)$. 
Therefore, inserting $p_m=\hat{p}_m$ into \eqref{eq:CGerror} gives 
\begin{equation}\label{eq:PCG_algebraic}
{\|\tilde u_h-R_mu_h\|_a
\leq\frac{\sqrt{\lambda\delta}}{2m+1}\|v\|_a.}
\end{equation}

{
To obtain the transient exponential bound, we set $\widehat S=S/\lambda$. Because  preconditioner rescaling does not change the Krylov space, we have }
\begin{equation*}
{\|\tilde u_h-R_mu_h\|_a=\min_{p_m(0)=1}\|p_m(\widehat{S}A)v\|_a\leq\|(I-\widehat{S}A)^mv\|_a.}
\end{equation*}
{By definition, we also obtain $\lambda_{\max}(\widehat{S}A)\leq1$ and} 
\begin{equation*}
{\|v\|_{\widehat S^{-1}}
=\sqrt\lambda\|v\|_{S^{-1}}
\leq\sqrt{\lambda\delta}\|v\|_a.}
\end{equation*}
{Then applying the smoothing estimate established in the proof of Theorem~\ref{thm:main} with the rescaled preconditioner $\widehat S$ yields}
\begin{equation}\label{eq:PCG_exp}
{\|\tilde u_h-R_mu_h\|_a
\leq\|(I-\widehat{S}A)^mv\|_a
\leq\varepsilon_m(\lambda\delta)\|v\|_a.}
\end{equation}
{Combining \eqref{eq:PCG_algebraic}/\eqref{eq:PCG_exp} with a triangle inequality completes the proof.
}
\end{proof}

With the help of assumptions about subspace decomposition of $\widetilde{V}$, we reformulate Theorem \ref{thm:PCG} as the next corollary.
\begin{corollary}\label{cor:PCG}
Let Assumptions \ref{assump:stable} and \ref{assump:regular} be true. Then for Algorithm \ref{alg:PCG} with $S=S_{\rm add}$ or $S=\bar{S}_{\rm mult}$, it holds that 
\begin{equation*}
\|u-R_mu_h\|_a\leq{2}\|u-\tilde{u}_h\|_a+{\varepsilon_m^{\rm PCG}(\widehat\delta)}\|u-u_h\|_a,
\end{equation*} 
where {$\widehat\delta=C_{\rm reg}C_{\rm st}$ for $S=S_{\rm add}$ and $\widehat\delta=C_{\rm st}M^2$ for $S=\bar{S}_{\rm mult}$}.
\end{corollary}
\begin{proof}
Recall that in the proof of Corollary \ref{cor:main}, it was shown that $\lambda_{\max}(S_{\rm add}A)\leq C_{\rm reg}$ and $\lambda_{\max}(\bar{S}_{\rm mult}A)\leq 1$. The rest of the proof is the same as in Corollary \ref{cor:main} and follows from Theorem \ref{thm:PCG}.
\end{proof}

\begin{remark}
{In practice, our algorithm applies to a wide class of FEs while Bank and Xu's smoothing \cite{BankXu2003b} focuses on superconvergent derivative recovery of two-dimensional Lagrange FEs. In theory, \cite{BankXu2003b} considered only Richardson smoothing, while our analysis covers general subspace correction smoothing, such as symmetrized GS, PCG, and block smoothing. The latter is essential for vector-valued FEs in $H(\text{curl})$ and $H(\text{div})$.}
\end{remark}

\section{Examples of Smoothing-based Postprocessing}\label{sect:examples}
In this section, we show the applications of Corollaries \ref{cor:main} and \ref{cor:PCG} to several model PDEs including the Poisson, curl-curl, and biharmonic equations. In the following, $\varepsilon_m$ denotes either the stationary smoothing factor $\varepsilon_m$ in Corollary \ref{cor:main} or the PCG smoothing factor $\varepsilon_m^{\rm PCG}$ in Corollary \ref{cor:PCG}, depending on the algorithm.

In the following applications, the subspace $\widetilde{V}_j$ consists of FE functions supported on local patches. A consequence is that $C_{\rm reg}=\mathcal{O}(1)$ in Assumption \ref{assump:regular} and $M=\mathcal{O}(1)$ in Corollary \ref{cor:main} are simply determined by the shape-regularity of the underlying mesh $\mathcal{T}_h$. Therefore, we shall not restate the verification of Assumption \ref{assump:regular} in proofs of the following theoretical results.

\subsection{Poisson Equation}\label{subsect:Poisson}
On a Lipschitz domain  $\Omega\subset\mathbb{R}^d$, we consider the Poisson boundary value problem 
\begin{equation}\label{eq:Poisson}
        -\Delta u=f\text{ in }\Omega,\qquad
        u=0\text{ on }\partial\Omega.
\end{equation}
For an integer $k\geq1$, let $\mathcal{P}_k$ be the space of polynomials of degree $\leq k$. Let $V\subset H_0^1(\Omega)$ be the FE space of continuous and piecewise $\mathcal{P}_k$ polynomials. Let $(\bullet,\bullet)=(\bullet,\bullet)_{\Omega}$ denote the $L^2(\Omega)$ inner product.  The FE method for \eqref{eq:Poisson} seeks $u_h\in V$ such that
\begin{equation}\label{eq:Poisson_FEM}
a(u_h,v_h)=(\nabla u_h,\nabla v_h)=(f,v_h), \quad \forall v_h\in V.
\end{equation} 

Consider the FE space of one higher polynomial degree 
\[
\widetilde{V}=\big\{v_h\in C(\overline{\Omega}): v_h|_T\in\mathcal{P}_{k+1}~\forall T\in\mathcal{T}_h,\, v_h|_{\partial\Omega}=0\big\}.
\]
Let $\{a_j\}_{1\leq j\leq\widetilde{N}}$ be the nodal set of $\mathcal{P}_{k+1}$ FE in $\mathcal{T}_h$ and $\phi_j\in \widetilde{V}$ the nodal basis function associated to $a_j$, i.e., $\phi_j(a_k)=\delta_{jk}$. Then $A$ corresponds to the $\mathcal{P}_{k+1}$ FE stiffness matrix $\mathbf{A}$, and $\tilde{f}\in \widetilde{V}^\prime$ is represented by the vector $\tilde{\mathbf{f}}=((f,\phi_i))_{1\leq i\leq\widetilde{N}}$. For the Poisson equation, we use the 1D or pointwise subspace decomposition  
\begin{equation}\label{eq:pointwise}
\widetilde{V}=\sum_{j=1}^{\widetilde{N}}\widetilde{V}_j,\qquad \widetilde{V}_j={\rm span}\{\phi_j\}. 
\end{equation}

Let $\mathbf{A}=\mathbf{D}-\mathbf{L}-\mathbf{L}^\top$ with $\mathbf{D}$ being the diagonal and $-\mathbf{L}$ being {the} lower triangular part of $\mathbf{A}$, respectively.
For the Poisson equation, the additive smoother $S_{\rm add}$ in Algorithm \ref{alg:Sa} based on \eqref{eq:pointwise} is reduced to a Jacobi preconditioner, which is represented by the matrix $\mathbf{S}_{\rm add}=\mathbf{D}^{-1}$ (denoted by  $S_{\rm add}\sim\mathbf{S}_{\rm add}$).  The symmetrized multiplicative smoother $\bar{S}_{\rm mult}$  in Algorithm \ref{alg:Sm} based on \eqref{eq:pointwise} is represented by the following matrix (denoted by $\bar{S}_{\rm mult}\sim\bar{\mathbf{S}}_{\rm mult}$)
\begin{align*}
\bar{\mathbf{S}}_{\rm mult}&=(\mathbf{D}-\mathbf{L})^{-1}+(\mathbf{D}-\mathbf{L})^{-\top}-(\mathbf{D}-\mathbf{L})^{-\top}\mathbf{A}(\mathbf{D}-\mathbf{L})^{-1}\\
&=(\mathbf{D}-\mathbf{L}^\top)^{-1}\mathbf{D}(\mathbf{D}-\mathbf{L})^{-1},    
\end{align*}
namely, one sweep of \emph{symmetrized} GS iteration, {which consists of one forward GS sweep followed by one backward GS sweep.}

By $A_1\lesssim A_2$ we mean $A_1\leq CA_2$ with $C>0$ being a generic uniform constant independent of $h$. Let $A_1\eqsim A_2$ denote $A_1\lesssim A_2\lesssim A_1$. We present an improved error estimate of continuous FEs for Poisson's equation in the next theorem.
\begin{theorem}\label{thm:Poisson_C0FEM}
Let $\Omega$ be a convex domain and $u_h\in V$ be the $\mathcal{P}_k$-FE solution for \eqref{eq:Poisson}. Let $R_mu_h$ be the output of either: (1) Algorithm \ref{alg:smoothing} with $S\sim \omega\mathbf{S}_{\rm add}$ or $\bar{\mathbf{S}}_{\rm mult}$; (2) Algorithm \ref{alg:PCG} with $S\sim\mathbf{S}_{\rm add}$ or $\bar{\mathbf{S}}_{\rm mult}$. It holds that 
    \begin{equation*}
        |u-R_mu_h|_{H^1(\Omega)}\lesssim h^{k+1}|u|_{H^{k+2}(\Omega)}+\varepsilon_mh^k|u|_{H^{k+1}(\Omega)}.
    \end{equation*}
\end{theorem}
\begin{proof}
Given $v\in \widetilde{V}\cap V^\perp$, let $v_j=v(a_j)\phi_j$ and $v=\sum_{j=1}^{\widetilde{N}}v_j$.  It follows from an inverse estimate and {an} element-wise homogeneity argument that 
\begin{equation}\label{eq:vj_splitting}
\sum_{j=1}^{\widetilde{N}}\|v_j\|_a^2\lesssim\sum_{j=1}^{\widetilde{N}}h^{-2}\|v_j\|_{L^2({\rm supp}(\phi_j))}^2\eqsim h^{-2}\Big\|\sum_{j=1}^{\widetilde{N}}v_j\Big\|_{L^2(\Omega)}^2.
\end{equation}
Let $P_h$ be the $a$-orthogonal projection onto the FE space $V$. On convex domains, the well-known  duality argument implies  
\begin{equation}\label{eq:duality}
\|v-P_hv\|_{L^2(\Omega)}\lesssim h|v-P_hv|_{H^1(\Omega)}.   
\end{equation}
Using \eqref{eq:vj_splitting} and \eqref{eq:duality},  we verify Assumption \ref{assump:stable} as follows:
\begin{align*}
\sum_{j=1}^{\widetilde{N}}\|v_j\|_a^2&\lesssim h^{-2}\|v\|_{L^2(\Omega)}^2=h^{-2}\|v-P_hv\|_{L^2(\Omega)}^2\\
&\lesssim|v-P_hv|_{H^1(\Omega)}^2=|v|_{H^1(\Omega)}^2.
\end{align*}
We then finish the proof by combining Corollaries \ref{cor:main} and \ref{cor:PCG} with a priori error estimates $|u-\tilde{u}_h|_{H^1(\Omega)}\lesssim h^{k+1}|u|_{H^{k+2}(\Omega)}$ and $|u-u_h|_{H^1(\Omega)}\lesssim h^k|u|_{H^{k+1}(\Omega)}$.
\end{proof}

In our analysis, improved accuracy is due to {the} smoothing error estimate $\|\tilde{u}_h-R_mu_h\|_a\leq\varepsilon_m\|\tilde{u}_h-u_h\|_a$, which is numerically illustrated by a $\mathcal{P}_1$-$\mathcal{P}_2$ FE pair (e.g., $\mathcal{P}_1$ for $V$ and $\mathcal{P}_2$ for $\widetilde{V}$) for Poisson's equation with exact solution $u(x,y) = \sin(\pi x)\sin(\pi y)$ on a unit square. The domain is partitioned into a uniform three-line mesh $\mathcal{T}_h$. As shown in Figure \ref{fig:P1P2_GS_CG}, the smoothing error decays rather quickly during the first four steps, but slowly thereafter. This phenomenon matches the piecewise definition of $\varepsilon_m$ in Theorem \ref{thm:main}.

\begin{figure}[ht]
    \centering
    \includegraphics[width=0.6\linewidth]{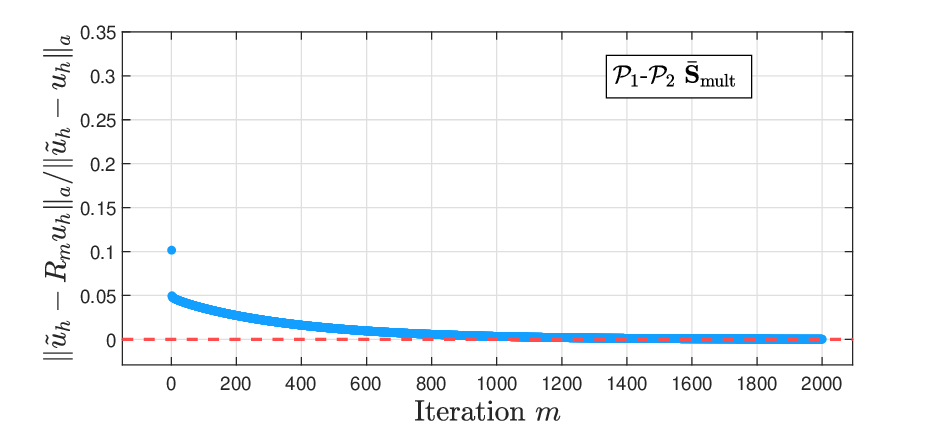}
    \includegraphics[width=0.6\linewidth]{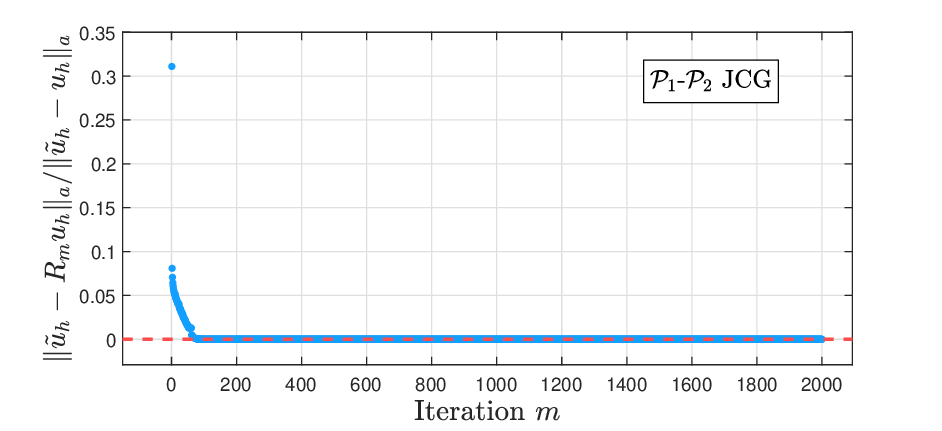}
    \caption{{Numerical relative smoothing error $\|\tilde u_h-R_mu_h\|_a/\|\tilde u_h-u_h\|_a$} for $\mathcal{P}_1$-$\mathcal{P}_2$ FEs for the Poisson equation on $\Omega=(0,1)^2$ with mesh-size $h=2^{-5}$. Top: {symmetrized GS smoother $\bar{\mathbf{S}}_{\rm mult}$} for $\mathcal{P}_2$-FE. Bottom: CG smoother with Jacobi preconditioner (JCG) for $\mathcal{P}_2$-FE.}
    \label{fig:P1P2_GS_CG}
\end{figure}

\subsubsection{Discontinuous Galerkin Method}\label{sect:discreteV} The smoothing-based framework also works for the discontinuous Galerkin (DG) method. Let $V=\{v_h\in L^2(\Omega): v_h|_T\in\mathcal{P}_k~\forall T\in\mathcal{T}_h\}$ and $\widetilde{V}=\{v_h\in L^2(\Omega): v_h|_T\in\mathcal{P}_{k+1}~\forall T\in\mathcal{T}_h\}$  be discontinuous FE spaces. Let $\mathcal{E}_h$ be the collection of $(d-1)$-simplices (e.g., edges for $d=2$ and faces for $d=3$) in $\mathcal{T}_h$. Given a penalty parameter $\gamma>0$, we define the bilinear form
\begin{align*}
a(u_h,v_h)
&=
\sum_{T\in\mathcal T_h}\int_T \nabla u_h\cdot\nabla v_h{\rm d}x-\sum_{E\in\mathcal E_h}\int_E\{\!\!\{\partial_nu_h\}\!\!\}\llbracket v_h\rrbracket {\rm d}s\\
&-\sum_{E\in\mathcal E_h}\int_E\{\!\!\{\partial_nv_h\}\!\!\}\llbracket u_h\rrbracket {\rm d}s
+
\sum_{E\in\mathcal E_h}\int_E \frac{\gamma}{h_E}\llbracket u_h\rrbracket\llbracket v_h\rrbracket {\rm d}s,
\end{align*}
where $h_E$ is the {diameter} of $E$, $\{\!\!\{\partial_nu_h\}\!\!\}$ and $\llbracket u_h\rrbracket$ denote the average of {the} normal derivative of $u_h$ and the jump of $u_h$ across an interior $E$, respectively (cf.~\cite[Chapter 10]{BrennerScott2008}). 
{On a boundary $E$, $\llbracket v\rrbracket=v$ and $\{\!\!\{\partial_n v\}\!\!\}=\partial_n v$ with $n$ being the outward unit normal. We take $k\geq1$ fixed and use the same penalty parameter in both spaces, chosen sufficiently large for $\widetilde V$.}

For the Poisson equation \eqref{eq:Poisson}, the $\mathcal{P}_k$-DG method is to find $u_h\in V$ such that
\begin{equation}\label{eq:DG}
    a(u_h,v_h)=(f,v_h),\quad\forall v_h\in V.
\end{equation}
By $\mathbf{S}_{\rm dG,a}$ and $\bar{\mathbf{S}}_{\rm dG,m}$ we denote the Jacobi and symmetrized GS {iterators} for the stiffness matrix $\mathbf{A}=\mathbf{A}_{\rm dG}$ of the $\mathcal{P}_{k+1}$-DG method, respectively. The error of DG methods is measured by the broken energy norm 
\begin{equation*}
\|v\|_{1,h}=\Big(\sum_{T\in\mathcal{T}_h}\|\nabla v\|_{L^2(T)}^2+\sum_{E\in\mathcal{E}_h}\gamma h_E^{-1}\|\llbracket v\rrbracket\|_{L^2(E)}^2\Big)^\frac{1}{2}.  
\end{equation*}
An improved estimate for the DG method is presented in the next theorem. 
\begin{theorem}
    Let $\Omega$ be a convex domain and $u_h\in V$ be the $\mathcal{P}_k$-DG solution in \eqref{eq:DG}. Assume that $\gamma$ is sufficiently large such that $a(v,v)\gtrsim\|v\|_{1,h}^2$ for all $v\in\widetilde{V}$. Let $R_mu_h$ be the output of either: (1) Algorithm \ref{alg:smoothing} with $S\sim\omega\mathbf{S}_{\rm dG,a}$ or $\bar{\mathbf{S}}_{\rm dG,m}$; (2) Algorithm \ref{alg:PCG} with $S\sim\mathbf{S}_{\rm dG,a}$ or $\bar{\mathbf{S}}_{\rm dG,m}$. Then we have
    \begin{equation*}
        \|u-R_mu_h\|_{1,h}\lesssim h^{k+1}|u|_{H^{k+2}(\Omega)}+\varepsilon_mh^k|u|_{H^{k+1}(\Omega)}.
    \end{equation*}
\end{theorem}
\begin{proof}
Note that $\|\bullet\|_a$ and $\|\bullet\|_{1,h}$ are equivalent on the DG space $\widetilde{V}$.
{For $v\in\widetilde V\cap V^\perp$, let $z\in H_0^1(\Omega)$ solve $-\Delta z=v$ and let $I_hz\in V\cap H_0^1(\Omega)$ be its continuous piecewise linear interpolant. Adjoint consistency, including the boundary-face terms under the convention above, and $a(v,I_hz)=0$ give
\[
\|v\|_{L^2(\Omega)}^2=a(v,z)=a(v,z-I_hz)
\lesssim h\|v\|_a\|z\|_{H^2(\Omega)}
\lesssim h\|v\|_a\|v\|_{L^2(\Omega)}.
\]
Here we used interpolation and trace estimates, coercivity, and elliptic regularity on the convex domain, as in the duality argument of \cite[Section 5.1]{ArnoldBrezziCockburnMarini2001}. Thus $\|v\|_{L^2(\Omega)}\lesssim h\|v\|_a$. For the elementwise nodal decomposition $v=\sum_jv_j$, inverse trace estimates then give
\[
\sum_j\|v_j\|_a^2\lesssim h^{-2}\sum_j\|v_j\|_{L^2(\Omega)}^2
\lesssim h^{-2}\|v\|_{L^2(\Omega)}^2
\lesssim\|v\|_a^2.
\]
This verifies Assumption~\ref{assump:stable}. The result follows from Corollaries~\ref{cor:main} and~\ref{cor:PCG}, norm equivalence on $\widetilde V$, and the standard DG approximation estimates.}
\end{proof}

\subsection{Maxwell Equation}
The smoothing-based framework could be applied to vector-valued FEs. The model problem is 
\begin{equation}\label{eq:Maxwell}
        \nabla\times\nabla\times u+u=f\text{ in }\Omega,\qquad \bm{n}\times (\nabla\times u)=0\text{ on }\partial\Omega.
\end{equation}
Here $\nabla\times$ is the curl operator and $\bm{n}$ is the outward unit normal to $\partial\Omega$.

Consider the $k$-th degree N\'ed\'elec edge FE of the first kind $\mathcal{N}d_k=\bm{x}\times[\mathcal{P}_{k-1}]^3+[\mathcal{P}_{k-1}]^3$.
Let $V$ be the N\'ed\'elec edge FE space of degree $\leq k$:
\[
V=\big\{v_h\in H({\rm curl},\Omega): v_h|_T\in\mathcal{N}d_k~\forall T\in\mathcal{T}_h\big\}.
\]
The FE solution  $u_h\in V\subset H({\rm curl},\Omega)$ is determined by
\begin{equation}\label{eq:Maxwell_FEM}
a(u_h,v_h)=(\nabla\times u_h,\nabla\times v_h)+(u_h,v_h)=(f,v_h),\quad \forall v_h\in V.
\end{equation}

Consider the one-order-higher FE space 
\[
\widetilde{V}=\big\{v_h\in H({\rm curl},\Omega): v_h|_T\in\mathcal{N}d_{k+1}~\forall T\in\mathcal{T}_h\big\}.
\]
The operator $A$ is represented by the $\mathcal{N}d_{k+1}$ FE stiffness matrix $\mathbf{A}=\mathbf{A}_{\rm Nd}$ of \eqref{eq:Maxwell_FEM}. It is shown in \cite{Zikatanov2008} that any splitting of $\widetilde{V}$ into the sum of 1D subspaces does not satisfy Assumption \ref{assump:stable}. A common remedy is the following subspace decomposition based on vertex-oriented patches
\begin{equation}\label{eq:block}
\widetilde{V}=\sum_{j=1}^{N_v}\widetilde{V}_j,\qquad \widetilde{V}_j=\big\{v_h\in \widetilde{V}: {\rm supp}(v_h)\subseteq\Omega_j\big\},
\end{equation}
where $N_v$ denotes the number of vertices in $\mathcal{T}_h$, and $\Omega_j$ is the union of elements sharing the $j$-th vertex in $\mathcal{T}_h$.

Given $r\in \widetilde{V}^\prime$, the action of the additive smoother $S_{\rm add}=S_{\rm blk,a}$ is given by $S_{\rm blk,a}r=\sum_{j=1}^{N_v}e_j$,
where each $e_j\in \widetilde{V}_j$ solves 
\[
(\nabla\times e_j,\nabla\times v_j)_{\Omega_j}+(e_j,v_j)_{\Omega_j}=\langle r,v_j\rangle,\quad \forall v_j\in \widetilde{V}_j.
\]
The action of $S_{\rm mult}=S_{\rm blk,m}$ is more complicated and is given by Algorithm \ref{alg:Sm} based on the decomposition \eqref{eq:block}. The symmetrized version $\bar{S}_{\rm blk,m}$ is given by \eqref{eq:Sbar_mult}. {Each symmetrized block GS step consists of one forward vertex-patch block GS sweep followed by one backward vertex-patch block GS sweep over the vertex patches.} 

In the literature, $S_{\rm blk,a}$ is referred to as an additive block smoother. Similarly, $S_{\rm blk,m}$ is referred to as a multiplicative block smoother. By convention, $S_{\rm blk, a}$ is a block Jacobi iterator and $S_{\rm blk, m}$ is a block GS iterator for the $\mathcal{N}d_{k+1}$ FE stiffness matrix $\mathbf{A}_{\rm Nd}$ for \eqref{eq:Maxwell_FEM}. 

The error estimates of N\'ed\'elec FE solutions are based on the norm $\|\bullet\|_{H(\rm curl,\Omega)}=(\|\bullet\|_{L^2(\Omega)}^2+\|\nabla\times\bullet\|_{L^2(\Omega)}^2)^{1/2}$ and semi-norm $|\bullet|_{H^k(\rm curl,\Omega)}=(|\bullet|_{H^k(\Omega)}^2+|\nabla\times\bullet|_{H^k(\Omega)}^2)^{1/2}$. The improved accuracy result for the N\'ed\'elec FE is presented in the next theorem. 
\begin{theorem}\label{thm:Maxwell}
    Let $\Omega$ be a convex domain and $u_h\in V$ be the $\mathcal{N}d_k$ FE solution in \eqref{eq:Maxwell_FEM}. Let $R_mu_h$ be the output of either: (1) Algorithm \ref{alg:smoothing} with $S=\omega S_{\rm blk,a}$ or $\bar{S}_{\rm blk,m}$; (2) Algorithm \ref{alg:PCG} with $S=S_{\rm blk, a}$ or $\bar{S}_{\rm blk, m}$. Then we have
    \begin{equation*}
        \|u-R_mu_h\|_{H({\rm curl},\Omega)}\lesssim h^{k+1}|u|_{H^{k+1}({\rm curl},\Omega)}+\varepsilon_mh^k|u|_{H^k({\rm curl},\Omega)}.
    \end{equation*}
\end{theorem}
\begin{proof}
{We first verify Assumption~\ref{assump:stable}. Let $X_h^r$ denote the continuous piecewise $\mathcal P_r$ space. For $v\in\widetilde V\cap V^\perp$, take the discrete Helmholtz decomposition
\[
v=r_h+\nabla q_h,\qquad r_h\in\widetilde V,\quad q_h\in X_h^{k+1},\quad \int_\Omega q_h=0,
\]
where $(r_h,\nabla s_h)=0$ for all $s_h\in X_h^{k+1}$. The argument of \cite[Proposition~4.4]{ArnoldFalkWinther2000} applies to the nested polynomial de Rham complexes associated with $V\subset\widetilde V$: both have the required first-order approximation estimates with mesh-size $h$, with constants depending on fixed $k$. Since $v\perp_a V$, it yields
\[
\|r_h\|_{L^2(\Omega)}\lesssim h\|v\|_a,
\qquad
\|q_h\|_{L^2(\Omega)}\lesssim h\|v\|_{L^2(\Omega)}.
\]
For fixed $k$, the $L^2$-stable vertex-patch decompositions in \cite{ArnoldFalkWinther2000} give
\[
r_h=\sum_jr_j,\quad q_h=\sum_jq_j,
\quad
\sum_j\bigl(\|r_j\|_{L^2(\Omega)}^2+\|q_j\|_{L^2(\Omega)}^2\bigr)
\lesssim\|r_h\|_{L^2(\Omega)}^2+\|q_h\|_{L^2(\Omega)}^2,
\]
where $r_j$ and $q_j$ are supported on $\Omega_j$. Setting
$v_j=r_j+\nabla q_j\in\widetilde V_j$ gives
\[
\sum_j\|v_j\|_a^2
\lesssim(1+h^{-2})\|r_h\|_{L^2(\Omega)}^2
+h^{-2}\|q_h\|_{L^2(\Omega)}^2
\lesssim\|v\|_a^2.
\]
Thus Assumption~\ref{assump:stable} holds, with a constant independent of $h$ for fixed $k$. We finish the proof by combining Corollaries~\ref{cor:main} and~\ref{cor:PCG} with the a priori approximation estimates  $\|u-\tilde{u}_h\|_{H({\rm curl},\Omega)}\lesssim h^{k+1}(|u|_{H^{k+1}(\Omega)}+|\nabla\times u|_{H^{k+1}(\Omega)})$ and $\|u-u_h\|_{H({\rm curl},\Omega)}\lesssim h^k(|u|_{H^k(\Omega)}+|\nabla\times u|_{H^k(\Omega)})$.}
\end{proof}

\begin{remark}
{A similar} improved error estimate still holds true for N\'ed\'elec FEs of the second kind. The block smoother in Theorem \ref{thm:Maxwell} can be replaced with any smoother that is able to ensure uniform contraction of a multigrid for the discrete H(curl) problem \eqref{eq:Maxwell_FEM}, e.g., the more efficient smoother due to Hiptmair \cite{Hiptmair1999SINUM}.
\end{remark}
   
To reduce the computational cost per smoothing step, we extract fine grid components from the Hiptmair-Xu (HX) preconditioner in \cite{HiptmairXu2007}. Let $\mathbf{D}_{\rm Nd}={\rm diag}(\mathbf{A}_{\rm Nd})$ and $\mathbf{D}_{k+1}={\rm diag}(\mathbf{A}_{k+1})$ with $\mathbf{A}_{k+1}$ being the $\mathcal{P}_{k+1}$-FE stiffness matrix for $(\nabla\bullet,\nabla\bullet)+(\bullet,\bullet)$. The HX smoother is
\begin{equation}\label{eq:HX_diag}
\mathbf{S}_\text{HX} = \mathbf{D}_{\rm Nd}^{-1} + \mathbf{G} \mathbf{D}_{k+1}^{-1}\mathbf{G}^\top + \mathbf{P}\vec{\mathbf{D}}_{k+1}^{-1} \mathbf{P}^\top,
\end{equation}
where $\mathbf{G}$ represents the discrete gradient from $\mathcal{P}_{k+1}$-FE to $\mathcal{N}d_{k+1}$-FE space, $\vec{\mathbf{D}}_{k+1}={\rm diag}(\mathbf{D}_{k+1},\mathbf{D}_{k+1},\mathbf{D}_{k+1})$, and $\mathbf{P}$ represents the N\'ed\'elec interpolation from $\mathcal{P}_{k+1}$-FE onto $\mathcal{N}d_{k+1}$-FE space. This construction yields an \emph{explicit} smoother, see Section~\ref{subsect:Numerical_Maxwell} for numerical results.

\subsection{Biharmonic Equation}\label{subsect:biharmonic} On a Lipschitz domain  $\Omega\subset\mathbb{R}^d$, we consider the biharmonic boundary value problem 
\begin{equation}\label{eq:biharmonic}
        \Delta^2 u=f\text{ in }\Omega,\qquad u=\partial_nu=0\text{ on }\partial\Omega.
\end{equation}
{For an integer $k\geq2$,} let $V$ and $\widetilde{V}$  be the same FE spaces as in {Section} \ref{subsect:Poisson} and define 
\begin{align*}
a(u_h,v_h)
&=
\sum_{T\in\mathcal T_h}\int_T \nabla^2 u_h:\nabla^2 v_h{\rm d}x-\sum_{E\in\mathcal E_h}\int_E\{\!\!\{\partial_n^2u_h\}\!\!\}\llbracket \partial_nv_h\rrbracket {\rm d}s\\
&-\sum_{E\in\mathcal E_h}\int_E\{\!\!\{\partial_n^2v_h\}\!\!\}\llbracket \partial_nu_h\rrbracket {\rm d}s
+
\sum_{E\in\mathcal E_h}\int_E \frac{\gamma}{h_E}\llbracket \partial_nu_h\rrbracket\llbracket \partial_nv_h\rrbracket {\rm d}s.
\end{align*}

{On boundary faces, the jump and average in this form are understood as the corresponding one-sided quantities. We use the same fixed penalty parameter in $V$ and $\widetilde V$, chosen sufficiently large for the augmented space.}

For the biharmonic equation \eqref{eq:biharmonic}, the continuous interior penalty method (denoted by CIP-$\mathcal{P}_k$, cf.~\cite{BrennerSung2005}) seeks $u_h\in V$ such that
\begin{equation}\label{eq:CIP}
    a(u_h,v_h)=(f,v_h),\quad\forall v_h\in V.
\end{equation}
For this problem, $A$ is represented by the CIP-$\mathcal{P}_{k+1}$ stiffness matrix $\mathbf{A}=\mathbf{A}_{\rm IP}=\mathbf{D}_{\rm IP}-\mathbf{L}_{\rm IP}-\mathbf{L}_{\rm IP}^\top$. By $\mathbf{S}_{\rm IP,a}=\mathbf{D}_{\rm IP}^{-1}$ and $\bar{\mathbf{S}}_{\rm IP,m}=(\mathbf{D}_{\rm IP}-\mathbf{L}_{\rm IP}^\top)^{-1}\mathbf{D}_{\rm IP}(\mathbf{D}_{\rm IP}-\mathbf{L}_{\rm IP})^{-1}$ we denote the Jacobi and symmetrized GS iterators for $\mathbf{A}_{\rm IP}$, respectively. Define the CIP energy norm 
\begin{equation*}
\|v\|_{2,h}=\Big(\sum_{T\in\mathcal{T}_h}\|\nabla^2 v\|_{L^2(T)}^2+\sum_{E\in\mathcal{E}_h}\gamma h_E^{-1}\|\llbracket \partial_nv\rrbracket\|_{L^2(E)}^2\Big)^\frac{1}{2}.  
\end{equation*}

Our {postprocessing} analysis requires full elliptic regularity \begin{equation}\label{eq:elliptic_regularity}
\|\Delta^{-2}g\|_{H^4(\Omega)}\lesssim\|g\|_{L^2(\Omega)}\quad \text{for any }g\in L^2(\Omega),
\end{equation}
where $\Delta^{-2}g$ solves \eqref{eq:biharmonic} with $f$ replaced with $g$.

\begin{theorem}\label{thm:CIP}
{Let $k\geq3$ and let $u_h\in V$ be the CIP-$\mathcal P_k$ solution of \eqref{eq:CIP}, with $u\in H^{k+2}(\Omega)$. Assume \eqref{eq:elliptic_regularity} and a sufficiently large fixed penalty parameter such that $a(v,v)\gtrsim\|v\|_{2,h}^2$ on $\widetilde V$.} Let $R_mu_h$ be the output of either: (1) Algorithm \ref{alg:smoothing} with $S\sim \omega\mathbf{S}_{\rm IP,a}$ or $\bar{\mathbf{S}}_{\rm IP,m}$; (2) Algorithm \ref{alg:PCG} with $S\sim\mathbf{S}_{\rm IP,a}$ or $\bar{\mathbf{S}}_{\rm IP,m}$. {Then}
\begin{equation*}
\|u-R_mu_h\|_{2,h}\lesssim h^k|u|_{H^{k+2}(\Omega)}+\varepsilon_m h^{k-1}|u|_{H^{k+1}(\Omega)}.
\end{equation*}
\end{theorem}
\begin{proof}
{Coercivity and inverse trace estimates give $\|v\|_a\eqsim\|v\|_{2,h}$ on $\widetilde V$. For $v\in\widetilde V\cap V^\perp$, its nodal decomposition satisfies
\[
\sum_j\|v_j\|_a^2\lesssim h^{-4}\|v\|_{L^2(\Omega)}^2.
\]
Let $z=\Delta^{-2}v$ and choose the nodal interpolant $I_hz\in V$. Consistency, Galerkin orthogonality, and element and face interpolation estimates yield
\[
\|v\|_{L^2(\Omega)}^2=a(v,z-I_hz)
\lesssim h^2\|v\|_{2,h}\|z\|_{H^4(\Omega)}
\lesssim h^2\|v\|_{2,h}\|v\|_{L^2(\Omega)}.
\]
Thus Assumption~\ref{assump:stable} holds. The result follows from Corollaries~\ref{cor:main} and~\ref{cor:PCG} and the standard CIP approximation estimates.}
\end{proof}

\begin{remark}
In general, \eqref{eq:elliptic_regularity} fails even on convex polygonal domains. On the other hand, \eqref{eq:elliptic_regularity} holds true on {a} domain with {a} smooth boundary $\partial\Omega$, while additional geometric variational crimes need to be analyzed in this case. Despite this limitation, Section~\ref{subsect:Numerical_biharmonic} demonstrates improved accuracy of $\|u-R_mu_h\|_{2,h}$.
\end{remark}

\section{Numerical Experiments}\label{sect:numerical}
In this section, we present numerical examples to demonstrate the efficiency of the proposed smoothing-based postprocessing. We apply Algorithms \ref{alg:smoothing} and \ref{alg:PCG} to the Poisson, Maxwell, biharmonic and Helmholtz equations. By ``$\mathcal{P}_k$-$\mathcal{P}_{k+1}$'' FE pair, we mean that a $\mathcal{P}_k$-FE solution $u_h$ is postprocessed via $m$-step smoothing to obtain a $\mathcal{P}_{k+1}$-FE function $R_m u_h$. The mesh sequence is generated by uniform refinement of an initial mesh except in {Sections} \ref{subsect:gmsh} and \ref{subsect:adapt}. {For uniform refinement, the reported orders are least-squares slopes of $\log(\mathrm{error})$ against $\log h$ over the last four meshes. The order convention on adaptive meshes is given in Section~\ref{subsect:adapt}.}

\subsection{Poisson Equation} \label{subsect:Numerical_Poisson}
Consider the Poisson equation \eqref{eq:Poisson} on the regular hexagon
$\Omega$ with vertices $(\sin\theta_i,\cos\theta_i)$, $\theta_i=i\pi/3$,
$i=0,1,\dots,5$. To satisfy the homogeneous Dirichlet boundary condition,
we choose the exact solution as
\begin{equation*}
u(x,y)=(3-4x^2)(3-(x+\sqrt{3}y)^2)(3-(x-\sqrt{3}y)^2).
\end{equation*}
We test improved accuracy of $R_mu_h$ based on $\mathcal{P}_1$-$\mathcal{P}_2$ and $\mathcal{P}_2$-$\mathcal{P}_3$ smoothing and three different smoothers: damped Jacobi, {symmetrized GS}, and conjugate gradient (CG).

For $\mathcal{P}_1$-$\mathcal{P}_2$ FE pairs, Table \ref{tab:P1P2} compares convergence of $|u-R_mu_h|_{H^1(\Omega)}$ using $m=1, 2, 3$ iterations of damped Jacobi (with $\omega=2/3$), {symmetrized GS}, and CG. For this problem, CG smoothing yields the best overall performance. We then visualize the behavior of CG for both $\mathcal{P}_1$-$\mathcal{P}_2$ and $\mathcal{P}_2$-$\mathcal{P}_3$ pairs. As shown in Figure \ref{fig:Poisson_Performance_All}, CG smoothing significantly improves the order of convergence of the FE solution $u_h$.

\begin{table}[htbp]
\centering
\small 
\setlength{\tabcolsep}{4.5pt} 
\caption{{Postprocessed errors} of $|u-R_mu_h|_{H^1(\Omega)}$ from $\mathcal{P}_1$-$\mathcal{P}_2$ pair for the Poisson equation on the hexagonal domain based on $m$-step damped Jacobi, {symmetrized GS} and CG smoothing, $m=1, 2, 3$. }
\resizebox{\linewidth}{!}{%
\begin{tabular}{llccccccc}
\toprule
\multirow{2}{*}{Method} & \multirow{2}{*}{$m$} & \multicolumn{6}{c}{$h$} & \multirow{2}{*}{Order} \\
\cmidrule(lr){3-8}
 & & $1/2^2$ & $1/2^3$ & $1/2^4$ & $1/2^5$ & $1/2^6$ & $1/2^7$ & \\
\midrule
$\mathcal{P}_1$ FE solution & 0 & 1.037e+1 & 5.259e+0 & 2.639e+0 & 1.321e+0 & 6.604e-1 & 3.302e-1 & 0.999 \\
\midrule
\multirow{3}{*}{Damped Jacobi} 
 & 1 & 2.715e+0 & 1.067e+0 & 4.602e-1 & 2.109e-1 & 1.005e-1 & 4.896e-2 & 1.077 \\
 & 2 & 2.037e+0 & 6.874e-1 & 2.466e-1 & 9.294e-2 & 3.688e-2 & 1.552e-2 & 1.330 \\
 & 3 & 1.875e+0 & 6.079e-1 & 2.081e-1 & 7.331e-2 & 2.619e-2 & 9.492e-3 & 1.485 \\
\midrule
\multirow{3}{*}{{\shortstack{Symmetrized\\Gauss--Seidel}}}
 & 1 & {2.162e+0} & {7.267e-1} & {2.667e-1} & {1.042e-1} & {4.303e-2} & {1.883e-2} & {1.275} \\
 & 2 & {1.975e+0} & {6.209e-1} & {2.112e-1} & {7.541e-2} & {2.749e-2} & {1.023e-2} & {1.456} \\
 & 3 & {1.876e+0} & {5.818e-1} & {1.933e-1} & {6.803e-2} & {2.441e-2} & {8.839e-3} & {1.483} \\
\midrule
\multirow{3}{*}{Conjugate Gradient} 
 & 1 & 2.313e+0 & 8.055e-1 & 2.836e-1 & 1.002e-1 & 3.545e-2 & 1.254e-2 & 1.500 \\
 & 2 & 1.822e+0 & 5.886e-1 & 2.002e-1 & 6.987e-2 & 2.459e-2 & 8.682e-3 & 1.509 \\
 & 3 & 1.653e+0 & 5.116e-1 & 1.694e-1 & 5.858e-2 & 2.057e-2 & 7.259e-3 & 1.514 \\
\bottomrule
\end{tabular}%
}\label{tab:P1P2}
\end{table}

\begin{figure}[!htbp]
    \centering
    \includegraphics[width=0.45\textwidth]{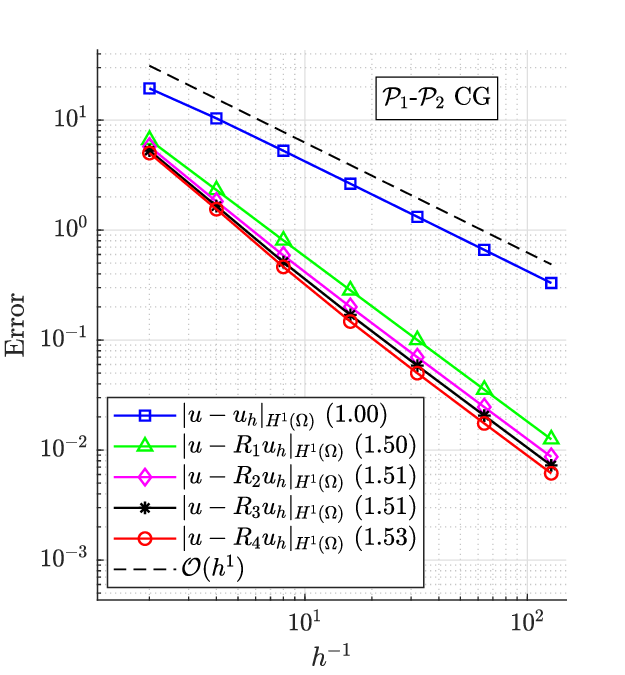}
    \includegraphics[width=0.45\textwidth]{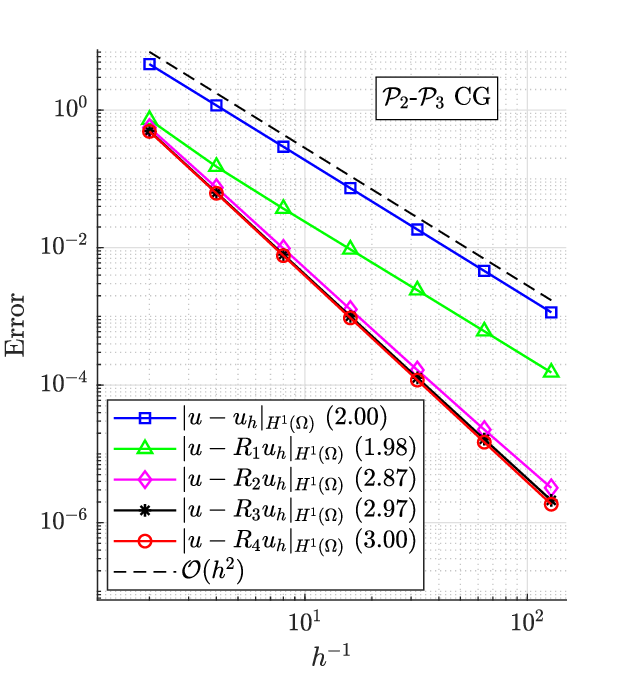}

    \caption{{Postprocessed errors} of $|u-R_mu_h|_{H^1(\Omega)}$ for the Poisson equation based on $m$-step CG smoothing, $m=1, 2, 3$. The order of convergence in $h$ is shown in parentheses.}
    \label{fig:Poisson_Performance_All}
\end{figure}


\subsection{Maxwell Equation}\label{subsect:Numerical_Maxwell} Consider the Maxwell equation \eqref{eq:Maxwell} on the unit cube $\Omega = (0,1)^3$ with exact solution
\begin{equation*}
u(x,y,z) = \pi^{-2} (\sin(\pi x)\cos(\pi y)\cos(\pi z),-\cos(\pi x)\sin(\pi y)\cos(\pi z), 0 ). 
\end{equation*}

We investigate the $\mathcal{N}d_1$-$\mathcal{N}d_2$ and $\mathcal{N}d_2$-$\mathcal{N}d_3$  FE pairs. To construct $R_m u_h$, we apply $m$ steps of four different smoothers: damped block Jacobi, {symmetrized} block GS, and PCG preconditioned by either a block Jacobi or an HX smoother in \eqref{eq:HX_diag}.

As observed in Tables \ref{tab:Nd1Nd2_smoother_compare} and \ref{tab:Nd2Nd3_smoother_compare}, even a simple 3-step smoothing procedure yields {an apparently} improved order of convergence for Nédélec FEs. For the smoothing methods under comparison, the {symmetrized} block GS iteration and PCG preconditioned by block Jacobi  deliver the best overall error reduction, albeit requiring {relatively greater} computational effort. Notably, the explicit HX smoother provides a satisfactory improvement in accuracy without solving patch-wise local problems.

\begin{table}[htbp]
\centering
\small
\setlength{\tabcolsep}{4.5pt}
\caption{{Postprocessed errors} of $\|u-R_mu_h\|_{H({\rm curl},\Omega)}$ from the $\mathcal{N}d_1$-$\mathcal{N}d_2$ pair for the Maxwell equation on $\Omega=(0,1)^3$ based on $m$-step {symmetrized} block GS and PCG smoothing, $m=1, 2, 3$.}
\resizebox{\linewidth}{!}{%
\begin{tabular}{llcccccc}
\toprule
\multirow{2}{*}{Method} & \multirow{2}{*}{$m$} & \multicolumn{5}{c}{$h$} & \multirow{2}{*}{Order} \\
\cmidrule(lr){3-7}
& & $1/2^1$ & $1/2^2$ & $1/2^3$ & $1/2^4$ & $1/2^5$ & \\
\midrule
$\mathcal{N}d_1$ FE solution & 0
& 1.685e-01 & 9.489e-02 & 4.939e-02 & 2.499e-02 & 1.254e-02 & 0.974 \\
\midrule
\multirow{3}{*}{{\shortstack{symmetrized\\block GS}}}
& 1 & 5.991e-02 & 1.936e-02 & 7.319e-03 & 2.409e-03 & 7.036e-04 & 1.595 \\
& 2 & 5.987e-02 & 1.753e-02 & 6.039e-03 & 2.152e-03 & 6.513e-04 & 1.574 \\
& 3 & 5.986e-02 & 1.734e-02 & 5.369e-03 & 1.986e-03 & 6.248e-04 & 1.582 \\
\midrule
\multirow{3}{*}{\shortstack{block Jacobi\\PCG}}
& 1 & 7.844e-02 & 3.460e-02 & 1.368e-02 & 5.727e-03 & 2.653e-03 & 1.237 \\
& 2 & 6.221e-02 & 2.446e-02 & 8.744e-03 & 2.640e-03 & 7.790e-04 & 1.665 \\
& 3 & 6.013e-02 & 1.968e-02 & 8.744e-03 & 2.440e-03 & 6.793e-04 & 1.641 \\
\midrule
\multirow{3}{*}{\shortstack{HX smoother\\PCG}}
& 1 & 1.029e-01 & 4.806e-02 & 2.177e-02 & 1.031e-02 & 5.056e-03 & 1.082 \\
& 2 & 8.000e-02 & 3.266e-02 & 1.208e-02 & 4.822e-03 & 2.182e-03 & 1.304 \\
& 3 & 7.102e-02 & 2.802e-02 & 9.275e-03 & 2.921e-03 & 1.028e-03 & 1.597 \\
\bottomrule
\end{tabular}%
}
\label{tab:Nd1Nd2_smoother_compare}
\end{table}

\begin{table}[htbp]
\centering
\small
\setlength{\tabcolsep}{4.5pt}
\caption{{Postprocessed errors} of $\|u-R_mu_h\|_{H({\rm curl},\Omega)}$ from $\mathcal{N}d_2$-$\mathcal{N}d_3$ pair for the Maxwell equation on $\Omega=(0,1)^3$ based on  $m$-step {symmetrized} block GS and PCG smoothing, $m=1, 2, 3$.}
\resizebox{\linewidth}{!}{%
\begin{tabular}{llcccccc}
\toprule
\multirow{2}{*}{Method} & \multirow{2}{*}{$m$} & \multicolumn{5}{c}{$h$} & \multirow{2}{*}{Order} \\
\cmidrule(lr){3-7}
& & $1/2^0$ & $1/2^1$ & $1/2^2$ & $1/2^3$ & $1/2^4$ & \\
\midrule
$\mathcal{N}d_2$ FE solution & 0
& 1.440e-01 & 5.928e-02 & 1.727e-02 & 4.541e-03 & 1.154e-03 & 1.897 \\
\midrule
\multirow{3}{*}{{\shortstack{symmetrized\\block GS}}}
& 1 & 9.354e-02 & 1.652e-02 & 2.350e-03 & 3.057e-04 & 3.898e-05 & 2.912 \\
& 2 & 9.354e-02 & 1.652e-02 & 2.348e-03 & 3.042e-04 & 3.851e-05 & 2.918 \\
& 3 & 9.354e-02 & 1.652e-02 & 2.348e-03 & 3.040e-04 & 3.845e-05 & 2.919 \\
\midrule
\multirow{3}{*}{\shortstack{block Jacobi\\PCG}}
& 1 & 9.652e-02 & 2.198e-02 & 4.503e-03 & 9.663e-04 & 2.257e-04 & 2.204 \\
& 2 & 9.377e-02 & 1.691e-02 & 2.545e-03 & 3.707e-04 & 6.033e-05 & 2.717 \\
& 3 & 9.354e-02 & 1.654e-02 & 2.368e-03 & 3.112e-04 & 4.041e-05 & 2.896 \\
\midrule
\multirow{3}{*}{\shortstack{HX smoother\\PCG}}
& 1 & 1.098e-01 & 3.212e-02 & 7.913e-03 & 2.021e-03 & 5.122e-04 & 1.988 \\
& 2 & 1.006e-01 & 2.315e-02 & 4.603e-03 & 1.010e-03 & 2.405e-04 & 2.196 \\
& 3 & 1.006e-01 & 2.015e-02 & 3.455e-03 & 6.305e-04 & 1.337e-04 & 2.416 \\
\bottomrule
\end{tabular}%
}
\label{tab:Nd2Nd3_smoother_compare}
\end{table}


\subsection{Biharmonic Equation}\label{subsect:Numerical_biharmonic}
We evaluate our approach on the biharmonic equation \eqref{eq:biharmonic} defined on $\Omega = (0,1)^2$. The exact solution is given by
\begin{equation*}
    u(x,y) = (1 - \cos(2\pi x))(1 - \cos(2\pi y)).
\end{equation*}
The problem is discretized using the CIP method \eqref{eq:CIP}. {This square-domain experiment violates the elliptic regularity assumption of Theorem~\ref{thm:CIP}.}

We investigate both the $\mathcal{P}_2$-$\mathcal{P}_3$ ($\gamma=10$) and $\mathcal{P}_3$-$\mathcal{P}_4$ ($\gamma=17$) FE pairs with {a} JCG smoother (CG with Jacobi preconditioner). It is observed in Figure \ref{fig:Biharmonic_Performance} that 3-4 steps of JCG smoothing yield {a clear accuracy improvement for the $\mathcal{P}_2$-CIP scheme, although Theorem \ref{thm:CIP} requires the polynomial order $k\geq3$. For the $\mathcal{P}_3$-$\mathcal{P}_4$ CIP pair, the improvement is smaller, and more than 10 JCG steps are needed to observe a higher convergence rate in this experiment.}

We then use PCG smoothing with {a} block Jacobi preconditioner corresponding to the vertex-oriented patch-wise space decomposition as in \eqref{eq:block}.
{Four steps of this smoother reduce the error and give a convergence rate about one order above that of the original FE solution over the tested meshes}, see Table \ref{tab:combined_pk_pk1_smoother_compare}.

\begin{figure}[!ht]
    \centering
    \includegraphics[width=0.49\textwidth]{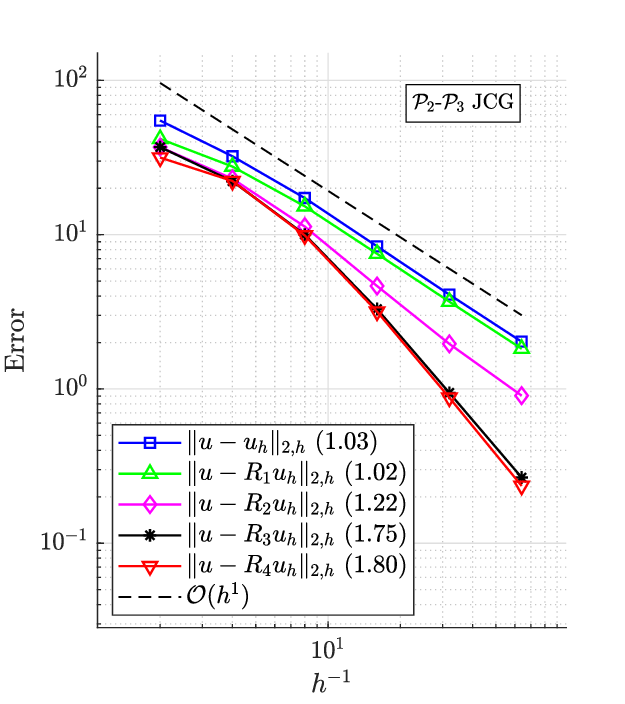}
    \includegraphics[width=0.49\textwidth]{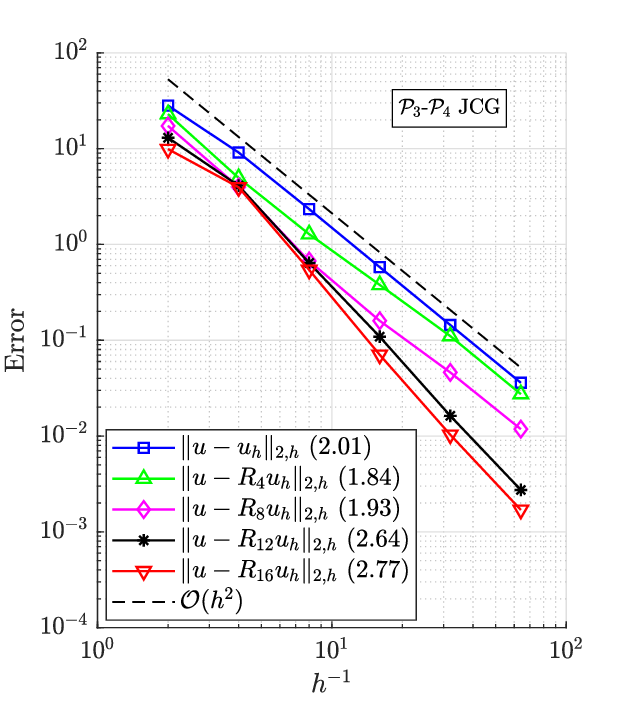}
    \caption{{Postprocessed errors} $\|u-R_mu_h\|_{2,h}$ from CIP methods for the biharmonic equation on $\Omega=(0,1)^2$ based on $m$-step JCG smoothing. Left: $\mathcal{P}_2$-$\mathcal{P}_3$, $m=1, 2, 3, 4;\ \gamma=10$; Right: $\mathcal{P}_3$-$\mathcal{P}_4$, $m=4, 8, 12, 16;\ \gamma=17$.  The {observed order} in $h$ is shown in parentheses.}
    \label{fig:Biharmonic_Performance}
\end{figure}

\begin{table}[!htbp]
\centering
\small
\setlength{\tabcolsep}{4.5pt}
\caption{{Postprocessed errors} $\|u-R_mu_h\|_{2,h}$ from the $\mathcal{P}_3$-$\mathcal{P}_4$ CIP (penalty parameter $\gamma=17$) pair for the biharmonic equation on  $\Omega=(0,1)^2$ based on $m$-step PCG smoothing with Jacobi or block Jacobi preconditioners, $m=1, 2, 3, 4$.}
\resizebox{\linewidth}{!}{%
\begin{tabular}{llccccccc}
\toprule
\multirow{2}{*}{Method} & \multirow{2}{*}{$m$} & \multicolumn{6}{c}{$h$} & \multirow{2}{*}{Order} \\
\cmidrule(lr){3-8}
& & $1/2^{1}$ & $1/2^{2}$ & $1/2^{3}$ & $1/2^{4}$ & $1/2^{5}$ & $1/2^{6}$ & \\
\midrule
$\mathcal{P}_{3}$ FE solution & 0 & 2.813e+1 & 9.108e+0 & 2.339e+0 & 5.793e-1 & 1.442e-1 & 3.598e-2 & 2.007 \\
\cmidrule(lr){1-9}
\multirow{4}{*}{\shortstack{$\mathcal{P}_{3}$-$\mathcal{P}_{4}$\\Jacobi PCG}} & 1 & 2.432e+1 & 7.481e+0 & 1.829e+0 & 4.415e-1 & 1.086e-1 & 2.697e-2 & 2.027 \\
 & 2 & 2.293e+1 & 6.833e+0 & 1.798e+0 & 4.867e-1 & 1.268e-1 & 3.212e-2 & 1.936 \\
 & 3 & 2.293e+1 & 6.833e+0 & 1.798e+0 & 4.374e-1 & 1.105e-1 & 2.741e-2 & 2.009 \\
 & 4 & 2.293e+1 & 4.946e+0 & 1.280e+0 & 3.782e-1 & 1.105e-1 & 2.741e-2 & 1.841 \\
\midrule
\multirow{4}{*}{\shortstack{$\mathcal{P}_{3}$-$\mathcal{P}_{4}$ block\\Jacobi PCG}} & 1 & 1.742e+1 & 4.991e+0 & 1.048e+0 & 2.297e-1 & 5.460e-2 & 1.340e-2 & 2.094 \\
 & 2 & 9.601e+0 & 3.536e+0 & 6.779e-1 & 1.463e-1 & 3.584e-2 & 8.947e-3 & 2.076 \\
 & 3 & 9.414e+0 & 2.752e+0 & 4.177e-1 & 4.894e-2 & 6.456e-3 & 1.093e-3 & 2.866 \\
 & 4 & 1.041e+1 & 2.752e+0 & 3.945e-1 & 4.306e-2 & 5.068e-3 & 7.192e-4 & 3.039 \\
\bottomrule
\end{tabular}%
}
\label{tab:combined_pk_pk1_smoother_compare}
\end{table}

\subsection{Helmholtz Equation} 
Although the theoretical part is devoted to SPD models,  we test  improved accuracy of $\mathcal{P}_1$-$\mathcal{P}_2$ and $\mathcal{P}_2$-$\mathcal{P}_3$ FEs for the Helmholtz equation 
    \begin{align*}
        \Delta u+\kappa^2u&=f\quad\text{ in }\Omega=(0,1)^2,\\
        \partial_nu-\texttt{i}\kappa u&=g\quad\text{ on }\partial\Omega,
    \end{align*}
a non-symmetric and indefinite problem. Here $\kappa$ is the wave number and $\texttt{i}=\sqrt{-1}$. FE errors are measured by the $\kappa$-weighted $H^1$-norm
$\|v\|_{1,\kappa} = ( \|\nabla v\|_{L^2(\Omega)}^2 + \kappa^2 \|v\|_{L^2(\Omega)}^2 )^{1/2}$.
The exact solution is chosen as
$u(x,y) = \exp(\texttt{i} \kappa (x+y)/\sqrt{2}).$

We use 4 steps of a Generalized Minimal Residual (GMRES) method with Jacobi {preconditioner} to {construct} $R_4u_h$. The improved accuracy phenomena under different wave numbers ($\kappa = \pi, 5\pi, 10\pi$) are illustrated in Figure \ref{fig:Helm_GMRES_All}. For larger wave numbers $\kappa$, the mesh must be sufficiently refined to enter the asymptotic convergence regime before the improved order of convergence becomes evident. Readers are referred to \cite{DuWuZhang2020} for a $\kappa$-explicit superconvergence analysis by polynomial preserving recovery. 

\begin{figure}[!ht]
    \centering
    \includegraphics[width=0.48\textwidth]{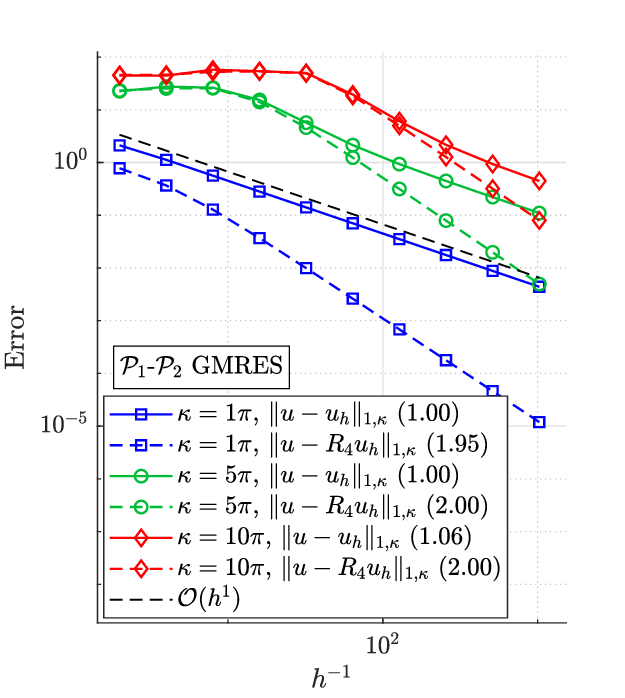} 
    \includegraphics[width=0.48\textwidth]{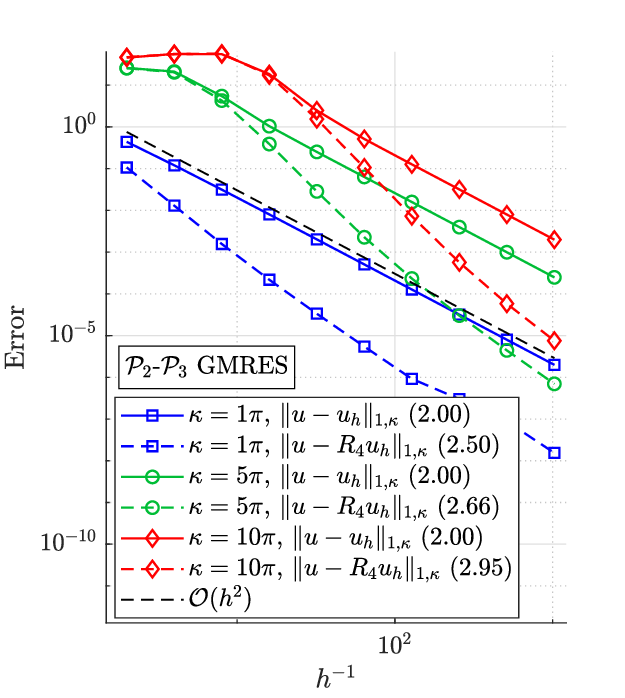} 
    \caption{{Postprocessed errors} of 4-step GMRES smoothing with Jacobi preconditioner for the Helmholtz equation. The order of convergence in $h$ is shown {in parentheses}.}
    \label{fig:Helm_GMRES_All}
\end{figure}

\subsection{Fully Unstructured Grids}\label{subsect:gmsh} To test improved accuracy on unstructured grids, given a mesh-size parameter $h>0$, we use the open-source package \texttt{Gmsh}  (cf.~\cite{GeuzaineRemacle2009}) to generate $\mathcal{T}_h$ without any local symmetry pattern, see Figure \ref{fig:mesh_gmesh}.
The corresponding results for the $\mathcal{P}_1$-$\mathcal{P}_2$ and $\mathcal{N}d_1$-$\mathcal{N}d_2$ pairs (using the exact settings from Sections \ref{subsect:Numerical_Poisson} and \ref{subsect:Numerical_Maxwell}) are presented in Table~\ref{tab:gmsh_superconvergence}. 
{In $\mathbb{R}^d$, the reported order is $q=-ds$, where $s$ is the least-squares slope of $\log(\mathrm{error})$ against $\log N$ over the last four meshes with $N$ being the number of elements.} The results indicate that improved accuracy {is} still present on fully unstructured grids. In addition, the block smoothers lead to more accurate postprocessed FE solutions than the pointwise ones. 

\begin{figure}[!htbp]
    \centering
    \includegraphics[width=0.45\textwidth]{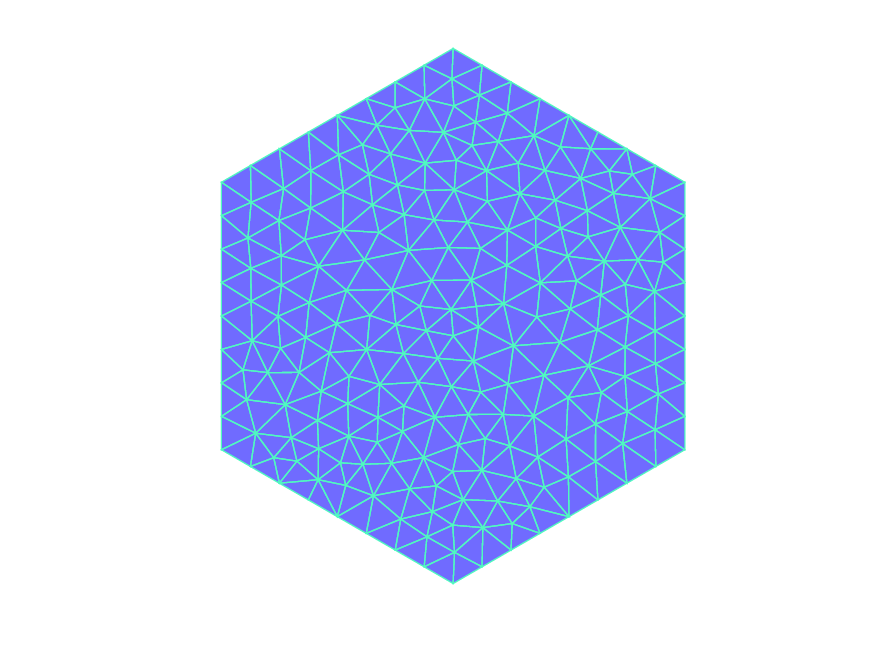}
    \includegraphics[width=0.45\textwidth]{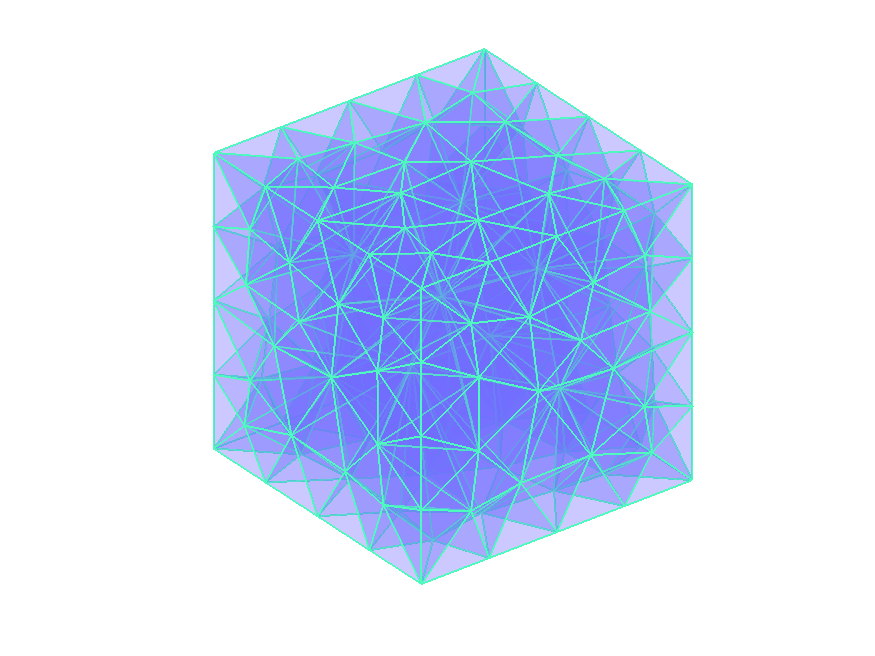}
    \caption{\texttt{Gmsh} meshes for regular hexagon, $h = 1/2^3$ (left) and unit cube, $h = 1/2^2$ (right).}
    \label{fig:mesh_gmesh}
\end{figure}

\begin{table}[htbp]
\centering
\small
\setlength{\tabcolsep}{4.5pt}
\caption{{Postprocessed errors} for the Poisson and Maxwell equations on \texttt{Gmsh} meshes.}
\resizebox{\linewidth}{!}{%
\begin{tabular}{llccccccc}
\toprule
\multirow{2}{*}{Method} & \multirow{2}{*}{$m$} & \multicolumn{6}{c}{$h$} & \multirow{2}{*}{Order} \\
\cmidrule(lr){3-8}
 & & $1/2^1$ & $1/2^2$ & $1/2^3$ & $1/2^4$ & $1/2^5$ & $1/2^6$ & \\
\midrule
\multicolumn{9}{c}{\textbf{Poisson equation} ($|u-R_m u_h|_{H^1(\Omega)}$)} \\
\midrule
$\mathcal{P}_1$ FE solution & 0 & 1.846e+01 & 1.024e+01 & 5.366e+00 & 2.617e+00 & 1.338e+00 & 6.671e-01 & 0.990 \\
\midrule
\multirow{3}{*}{Jacobi PCG} 
 & 1 & 8.888e+00 & 2.883e+00 & 1.306e+00 & 5.821e-01 & 2.909e-01 & 1.381e-01 & 1.063 \\
 & 2 & 6.691e+00 & 1.965e+00 & 7.547e-01 & 2.679e-01 & 1.220e-01 & 5.329e-02 & 1.250 \\
 & 3 & 6.230e+00 & 1.756e+00 & 6.147e-01 & 2.059e-01 & 8.640e-02 & 3.635e-02 & 1.338 \\
 \midrule
\multirow{3}{*}{\shortstack{Block Jacobi\\ PCG}} 
 & 1 & 6.922e+00 & 2.543e+00 & 1.194e+00 & 5.376e-01 & 2.709e-01 & 1.326e-01 & 1.041 \\
 & 2 & 5.697e+00 & 1.526e+00 & 5.131e-01 & 1.732e-01 & 7.258e-02 & 3.200e-02 & 1.315 \\
 & 3 & 5.697e+00 & 1.399e+00 & 4.580e-01 & 1.428e-01 & 5.333e-02 & 2.219e-02 & 1.440 \\
\midrule
\multicolumn{9}{c}{\textbf{Maxwell equation} ($\|u-R_mu_h\|_{H({\rm curl},\Omega)}$)} \\
\midrule
$\mathcal{N}d_1$ FE solution & 0 & 1.767e-01 & 9.317e-02 & 5.233e-02 & 2.678e-02 & 1.353e-02 & -- & 1.002 \\
\midrule
\multirow{3}{*}{\shortstack[l]{block Jacobi\\PCG}}
 & 1 & 8.153e-02 & 3.491e-02 & 1.502e-02 & 6.526e-03 & 3.149e-03 & -- & 1.247 \\
 & 2 & 4.917e-02 & 2.627e-02 & 9.063e-03 & 2.969e-03 & 1.153e-03 & -- & 1.626 \\
 & 3 & 3.286e-02 & 2.307e-02 & 7.853e-03 & 2.345e-03 & 7.099e-04 & -- & 1.808 \\
\midrule
\multirow{3}{*}{\shortstack[l]{HX smoother\\PCG}}
 & 1 & 1.159e-01 & 5.050e-02 & 2.461e-02 & 1.153e-02 & 5.699e-03 & -- & 1.132 \\
 & 2 & 8.918e-02 & 3.491e-02 & 1.448e-02 & 6.151e-03 & 2.944e-03 & -- & 1.282 \\
 & 3 & 7.164e-02 & 2.920e-02 & 1.049e-02 & 3.947e-03 & 1.767e-03 & -- & 1.454 \\
\bottomrule
\end{tabular}%
}
\label{tab:gmsh_superconvergence}
\end{table}

\subsection{Adaptive FE by Smoothing}\label{subsect:adapt}
Finally, we use $\|u_h-R_mu_h\|_a$ as {an} a posteriori error estimate for adaptive mesh refinement. Consider the Poisson equation \eqref{eq:Poisson} on the L-shaped domain $\Omega=(-1,1)^2\backslash([0,1)\times[-1,0))$ with the exact solution $u=\phi(r)r^{2/3}\sin(2\theta/3)$, where $(r,\theta)$ {are the polar coordinates} near the origin and $\phi(r) = 
    ( 1 - r/0.9 )^8\mathbf{1}_{\{r\leq0.9\}}$ is a cutoff function.

We adopt the $\mathcal{P}_k$-$\mathcal{P}_{k+1}$ strategy with $m=4$ steps of CG smoothing. {Let $\eta_T:=|R_4u_h-u_h|_{H^1(T)}$ and $\eta_h^2:=\sum_{T\in\mathcal T_h}\eta_T^2$. We use newest vertex bisection and D\"orfler marking $\sum_{T\in\mathcal M_h}\eta_T^2\geq0.5\eta_h^2$ for adaptive mesh refinement. With $N=\dim V$, the reported order is $q=-2s$, where $s$ is the least-squares slope of $\log(\mathrm{error})$ against $\log N$ over the last 30 meshes. In Table \ref{tab:iteration_detailed_errors}, the effectivity ratios approach $1$, although the present quasi-uniform analysis does not establish improved FE accuracy and asymptotic exactness of smoothing-based estimators under local refinement.}


\begin{table}[htbp]
\centering
\small
\setlength{\tabcolsep}{4.5pt} 
\caption{Adaptive $\mathcal{P}_k$-FE errors for the Poisson equation based on postprocessing-type  $\mathcal{P}_k$-$\mathcal{P}_{k+1}$ a posteriori error estimate (CG smoother, $m=4$).}
\label{tab:iteration_detailed_errors}
\resizebox{\linewidth}{!}{%
{
\begin{tabular}{llccccccc} %
\toprule
FE & Metric & \multicolumn{6}{c}{Degrees of freedom} & Order \\ %
\midrule
\multirow{5}{*}{$\mathcal{P}_1$} 
 & $\mathrm{DoFs}$                                      & 45        & 149       & 722       & 3648      & 16685     & 71452     & -- \\
\cmidrule(l){2-9}
 & $|u-R_4 u_h|_{H^1(\Omega)}$                         & 4.195e-02 & 1.186e-02 & 2.434e-03 & 6.005e-04 & 1.686e-04 & 5.593e-05 & 1.638 \\
\cmidrule(l){2-9} %
 & $\eta_h$                                              & 8.575e-02 & 4.112e-02 & 1.753e-02 & 7.698e-03 & 3.559e-03 & 1.714e-03 & 1.012 \\
 & $|u-u_h|_{H^1(\Omega)}$                               & 9.154e-02 & 4.226e-02 & 1.765e-02 & 7.715e-03 & 3.563e-03 & 1.715e-03 & 1.014 \\
 & $\eta_h/|u-u_h|_{H^1(\Omega)}$                      & 0.9368    & 0.9730    & 0.9931    & 0.9978    & 0.9991    & 0.9995    & -- \\
\midrule
\multirow{5}{*}{$\mathcal{P}_2$} 
 & $\mathrm{DoFs}$                                      & 163       & 525       & 1789      & 6707      & 24609     & 86605     & -- \\
\cmidrule(l){2-9}
 & $|u-R_4 u_h|_{H^1(\Omega)}$                         & 1.255e-02 & 1.975e-03 & 2.815e-04 & 4.574e-05 & 8.851e-06 & 1.970e-06 & 2.561 \\
\cmidrule(l){2-9} 
 & $\eta_h$                                              & 2.665e-02 & 7.853e-03 & 2.035e-03 & 5.290e-04 & 1.413e-04 & 3.990e-05 & 2.032 \\
 & $|u-u_h|_{H^1(\Omega)}$                               & 2.777e-02 & 7.937e-03 & 2.046e-03 & 5.302e-04 & 1.415e-04 & 3.994e-05 & 2.034 \\
 & $\eta_h/|u-u_h|_{H^1(\Omega)}$                        & 0.9597    & 0.9894    & 0.9946    & 0.9977    & 0.9985    & 0.9989    & -- \\
\midrule
\multirow{5}{*}{$\mathcal{P}_3$} 
 & $\mathrm{DoFs}$                                      & 811       & 1603      & 4327      & 12361     & 35218     & 95644     & -- \\
\cmidrule(l){2-9}
 & $|u-R_4 u_h|_{H^1(\Omega)}$                         & 1.335e-03 & 2.311e-04 & 3.368e-05 & 4.606e-06 & 7.035e-07 & 1.341e-07 & 3.534 \\
\cmidrule(l){2-9} 
 & $\eta_h$                                              & 3.379e-03 & 8.776e-04 & 1.807e-04 & 3.460e-05 & 6.987e-06 & 1.529e-06 & 3.068 \\
 & $|u-u_h|_{H^1(\Omega)}$                               & 3.510e-03 & 8.943e-04 & 1.828e-04 & 3.481e-05 & 7.014e-06 & 1.534e-06 & 3.072 \\
 & $\eta_h/|u-u_h|_{H^1(\Omega)}$                        & 0.9628    & 0.9813    & 0.9885    & 0.9939    & 0.9962    & 0.9965    & -- \\
\midrule
\multirow{5}{*}{$\mathcal{P}_4$} 
 & $\mathrm{DoFs}$                                      & 1769      & 3057      & 6533      & 15129     & 35705     & 83633     & -- \\
\cmidrule(l){2-9}
 & $|u-R_4 u_h|_{H^1(\Omega)}$                         & 3.326e-04 & 4.651e-05 & 7.071e-06 & 8.795e-07 & 1.251e-07 & 1.888e-08 & 4.579 \\
\cmidrule(l){2-9} 
 & $\eta_h$                                              & 5.683e-04 & 1.328e-04 & 2.742e-05 & 4.615e-06 & 7.900e-07 & 1.384e-07 & 4.124 \\
 & $|u-u_h|_{H^1(\Omega)}$                               & 6.210e-04 & 1.371e-04 & 2.790e-05 & 4.671e-06 & 7.979e-07 & 1.395e-07 & 4.131 \\
 & $\eta_h/|u-u_h|_{H^1(\Omega)}$                        & 0.9153    & 0.9681    & 0.9828    & 0.9880    & 0.9901    & 0.9920    & -- \\
\bottomrule
\end{tabular}%
}
}
\end{table}

\begin{remark}
{When $m=1$, the $\mathcal{P}_1$-$\mathcal{P}_2$  error indicator $|u_h-R_1u_h|_{H^1(T)}$ is a $p$-variant of the $h$-$h/2$ smoother-type error estimator in \cite{LiShui2026}, while those error indicators are not asymptotically exact in practice. In fact, \cite{LiShui2026} mainly focused on a posteriori upper and lower error bounds under an $h/2$-space $\widetilde{V}$ and a single smoothing step. In contrast, $|u_h-R_mu_h|_{H^1(T)}$ exhibits practical asymptotic exactness by exploiting the smoothing-based improved accuracy mechanism under multiple smoothing passes. }
\end{remark}

\begin{remark}
Our smoothing-based a posteriori error estimate is different from the smoothed adaptive FE method (S-AFEM) in \cite{MulitaGianiHeltai2021}. In particular, S-AFEM replaces exact solves of most algebraic linear systems arising from {the} adaptive feedback loop with simple smoothing passes, while the local error indicator in \cite{MulitaGianiHeltai2021} is of {residual type}.
\end{remark}

\section{Concluding Remarks}\label{sect:conclusion}
We proposed a smoothing-based postprocessing for general FEs in arbitrary space dimension. For additive and multiplicative smoothers, we established improved accuracy of FE discretizations of SPD problems on quasi-uniform grids. However, a direct generalization of our smoothing approach fails for saddle-point problems such as mixed FE discretizations of Darcy and Stokes equations, which is the focus of our future work.

\bibliographystyle{siamplain}

\end{document}